\documentclass[secthm,seceqn,amsthm,ussrhead,12pt]{amsart}
\usepackage{amsmath,latexsym}
\usepackage[english]{babel}
\usepackage[psamsfonts]{amssymb}
\usepackage{times}
\usepackage{cite}

\usepackage[mathcal]{euscript}
\numberwithin{equation}{section} \textwidth=15.5cm
\newtheorem{theorem}{Theorem}[section]

\newtheorem{corollary}[theorem]{Corollary}

\newtheorem{lemma}[theorem]{Lemma}

\newtheorem{problem}[theorem]{Problem}
\newtheorem{proposition}[theorem]{Proposition}

\theoremstyle{definition}
\newtheorem{remark}[theorem]{Remark}

\usepackage[colorlinks, bookmarks=true]{hyperref}
\usepackage{color,graphicx,shortvrb}

\usepackage{enumerate}
\usepackage{amsmath}
\usepackage{amssymb}

\usepackage[latin 1]{inputenc}

\def\J#1#2#3{ \left\{ #1,#2,#3 \right\} }

\def\11{\textbf{$1$}}

\begin{document}
\numberwithin{equation}{section}

\title[boundedness of measures and 2-local triple derivations]{Boundedness of completely additive measures with application to 
 2-local triple derivations}

\author[Hamhalter]{Jan Hamhalter}
\email{hamhalte@math.feld.cvut.cz}
\address{Czech Technical University, Faculty of Electrical Engineering, Technicka 2, 166 27, Prague 6,
Czech Republic}

\author[Kudaybergenov]{Karimbergen Kudaybergenov}
\email{karim2006@mail.ru}
\address{Ch. Abdirov 1, Department of Mathematics, Karakalpak State University, Nukus 230113, Uzbekistan}

\author[Peralta]{Antonio M. Peralta}
\email{aperalta@ugr.es}
\address{Departamento de An{\'a}lisis Matem{\'a}tico, Facultad de
Ciencias, Universidad de Granada, 18071 Granada, Spain.}

\author[Russo]{Bernard Russo}
\email{brusso@uci.edu}
\address{Department of Mathematics, UC Irvine, Irvine CA, USA}

\thanks{First named author  was supported by  the ``Grant Agency of the Czech
Republic" grant number P201/12/0290, ``Topological and geometrical
properties of Banach spaces and operator algebras".
Third named author was partially supported by the Spanish Ministry of Economy and
 Competitiveness project no. MTM2014-58984-P and Junta de Andaluc\'{\i}a
 grant FQM375}



\date{}

\begin{abstract} We prove a Jordan version of  Dorofeev's boundedness theorem for completely additive measures  and use it to show that every {\rm(}not necessarily linear nor continuous{\rm)} 2-local triple derivation on a continuous JBW$^*$-triple is a triple derivation.\end{abstract}

\maketitle

\maketitle
\thispagestyle{empty}

\section{Introduction and Background}\label{sec:intro}

Let $\mathcal{P} (M)$ denote the lattice of projections in a von
Neumann algebra $M.$
Let $X$ be a Banach space. A mapping $\mu:
\mathcal{P} (M)\to X$ is said to be \emph{finitely additive} when
\begin{equation}\label{2.1}
\mu \left(\sum\limits_{i=1}^n p_i\right) = \sum\limits_{i=1}^{n}
\mu (p_i),
\end{equation}
for every family $p_1,\ldots, p_n$ of mutually orthogonal
projections in $M.$ A mapping $\mu: \mathcal{P} (M)\to X$ is said to be \emph{bounded} when the set
$$
\left\{ \|\mu (p)\|: p \in \mathcal{P} (M) \right\}
$$ is bounded.\smallskip

The celebrated Bunce-Wright-Mackey-Gleason theorem (\cite{BuWri92}, \cite{BuWri94}) states that if $M$ has no summand of type $I_2$, then every bounded finitely additive mapping $\mu: \mathcal{P} (M)\to X$ extends to a bounded linear operator from $M$ to $X$.\smallskip

Answering a question posed by George Mackey,  Gleason's original  theorem \cite{Gleason57}  characterizes quantum mechanical states on a separable Hilbert space in terms of density operators, and thus plays an important role in the foundations of quantum mechanics. The interdisciplinary nature of the Bunce-Wright-Mackey-Gleason theorem makes this result very useful in a wide range of  topics. Applications can be found in quantum physics and quantum information (cf. \cite{Dvu1993}, \cite{Wright98}, \cite{RuWrigh00}, \cite{Molnar2001}, \cite[Chapter 7]{Ham03}, and \cite{CooJunNavPerVill2013}, among many others), and in  functional analysis with studies on vector-valued measures on von Neumann algebras and 2-local maps on von Neumann algebras, JBW$^*$-algebras and JBW$^*$-triples (see \cite{EdRu99}, \cite{AyuKudPos}, \cite{AyuKudPer}, \cite{BuFerGarPe2015RACSAM} \cite{BuFerGarPe2015JMAA} and \cite{KOPR2014}).
\smallskip

According to the terminology employed in \cite{Shers2008} and \cite{Doro}, a completely
additive mapping $\mu : \mathcal{P} (M)\to \mathbb{C}$---that is, (\ref{2.1}) holds with $X=\mathbb{C}$ for an arbitrary set of mutually orthogonal projections,  is called a \emph{charge}.
The  Dorofeev--Sherstnev
 theorem (\cite[Theorem 29.5]{Shers2008} or \cite[Theorem 2]{Doro})  states that any charge on a  von Neumann algebra with no summands of type $I_n$ is bounded.\smallskip

The  Dorofeev-Shertsnev theorem was used in \cite{KOPR2014} in order   to apply the Bunce-Wright-Mackey-Gleason theorem to prove the main result of that paper, namely, that a 2-local triple derivation on a von Neumann algebra is a triple derivation (see the next subsection).  In section~\ref{susect: Dorofeev} of this paper, we shall establish the first main result of this paper, namely,  a Jordan version  of Dorofeev's boundedness theorem (Theorem~\ref{t Jordan Dorofeev-Sherstnev}).   This will be used in section~\ref{section5} to  show that 2-local triple derivations on certain continuous JW$^*$-algebras are triple derivations (Theorem~\ref{t continuous type 2}).  Combined with the main result of section~\ref{section3} (Theorem~\ref{prop}), this will prove the second main result of this paper, namely,  that every 2-local triple derivation on an arbitrary continuous JBW$^*$-triple is a triple derivation (Theorem~\ref{continuous}). \smallskip


Having described the contents and potential impact of this paper, we shall now present more background and some preliminary material.  \smallskip

 We shall use the term \emph{measure} to denote a complex valued finitely additive function $\mu$ on the projections of a von Neumann algebra or a JBW$^*$-algebra.    If $\mu$ is positive  (resp. real) valued, we call it a positive (resp. signed) measure.
If countable additivity or complete additivity is assumed, it will be explicitly stated.\smallskip

Let us recall that a \emph{derivation} is a linear map $D$  from an algebra $A$  to a two sided $A$-module $M$ over the algebra satisfying the Leibniz identity:   $D(ab)=a\cdot D(b)+D(a)\cdot b$ for all $a,b\in A$.
\smallskip

Local derivations were introduced simultaneously in 1990 by Kadison \cite{Kad90} and by Larson-Sourour \cite{LarSou}.  A \emph{local derivation} from an algebra into a module is a linear mapping whose value at each point in the algebra coincides with the value of some derivation at that point.   Kadison proved that every continuous local derivation of a von Neumann algebra into a dual Banach module is in fact a derivation.  Johnson \cite{John01} extended Kadison's result to C$^*$-algebras, and moreover showed that the continuity assumption was not necessary.  Larson and Sourour showed that a local derivation on the algebra of all bounded linear operators on a  Banach space is a derivation. \smallskip

Let us also recall that a \emph{triple derivation} is a linear map $D$  from a triple system $E$  to an  $E$-module $N$ over the triple system satisfying the triple Leibniz identity:   $D(\{abc\})=\{D(a)bc\}+\{a D(b) c\}+\{abD(c)\}$ for all $a,b,c\in E$, where $\{abc\}$ denotes the triple product. (Jordan triple systems are defined later in this section.)
\smallskip

Local triple derivations were introduced in 2013 by Michael Mackey \cite{Mack}.
A \emph{local triple derivation} on a triple system  is a linear mapping whose value at each point in the triple system coincides with the value of some triple derivation at that point.
Mackey showed that a continuous local triple derivation on a JBW$^*$-triple (to itself) is a triple derivation, an exact analog of Kadison's result mentioned above. This result was extended to JB$^*$-triples in 2014 by Burgos, Fernandez-Polo, and Peralta \cite{BurPolPerBLMS14}, who also showed that the continuity assumption was redundant, an exact analog of Johnson's result also mentioned above.
\smallskip

Since 1997 there has been much interest in the notion of 2-local derivation  and more recently, in the notion of 2-local triple derivation.  The application of the main theorem of this paper concerns 2-local triple derivations. A \emph{2-local derivation} (respectively, \emph{2-local triple derivation}) from an algebra (respectively, triple system) into itself is a  mapping (not necessarily linear) whose values at each pair of points in the algebra (respectively, triple system) coincides with the values of some derivation (respectively, triple derivation) at those  two points. 2-local derivations were introduced in 1997 by Semrl \cite{Semrl97} and 2-local triple derivations were introduced in 2014 by Kudaybergenov, Oikhberg, Peralta, and Russo \cite{KOPR2014} although the concept was mentioned by Michael Mackey in a lecture  in 2012  at a conference in Hong Kong celebrating the 65th birthday of Cho-Ho Chu.
 It is now known that, for von Neumann algebras, a 2-local derivation  is in fact a derivation (Ayupov-Kudaybergenov  \cite{AyuKudPos}) and, as noted above,  a 2-local triple derivation  is a triple derivation (Kudaybergenov-Oikhberg-Peralta-Russo \cite{KOPR2014}).\smallskip

For an elaboration of the above summary, see the forthcoming survey of Ayupov, Kudaybergenov, and Peralta, \cite{AyuKudPer}.
Local and 2-local derivations have also been considered on algebras of measurable operators associated with von Neumann algebras. For more details on this, see the forthcoming survey of Ayupov and Kudaybergenov \cite{AyuKud}.\smallskip


A complex  \emph{Jordan triple} is a complex
vector space $E$ equipped with a non-trivial  triple
product $$ E \times E \times E \rightarrow E$$
$$(x,y,z) \mapsto \J xyz $$
which is bilinear and symmetric in the outer variables and
conjugate linear in the middle one satisfying the
so-called \emph{``Jordan Identity''}:
$$L(a,b) L(x,y) -  L(x,y) L(a,b) = L(L(a,b)x,y) - L(x,L(b,a)y),$$
for all $a,b,x,y$ in $E$, where $L(x,y) z := \J xyz$. \smallskip

A subspace $F$ of a Jordan triple $E$ is said to be a
\emph{subtriple} if $\J FFF \subseteq F$ and an \emph{ideal} if
$\{E,E,J\}+\{E,J,E\} \subseteq J.$ \smallskip

A \emph{(complex) JB$^*$-triple} is a complex Jordan Banach triple
${E}$ satisfying the following axioms: \begin{itemize}
 \item For each $a$ in ${E}$ the map $L(a,a)$ is an hermitian
operator on $E$ with non negative spectrum; \item  $\left\|
\{a,a,a\}\right\| =\left\| a\right\| ^3$ for all $a$ in ${A}.$
\end{itemize}\smallskip

A JB$^*$-algebra is a complex Jordan Banach algebra ($A,\circ$) equipped
with an algebra involution $^*$ satisfying  $\|\J a{a^*}a \|= \|a\|^3$, $a\in
A$.  (Recall that $\J a{a^*}a  =
 2 (a\circ a^*) \circ a - a^2 \circ a^*$.)     JB-algebras are precisely the self adjoint parts
of JB$^*$-algebras, and a JBW-algebra is a JB-algebra which is
a dual space.\smallskip

Every C$^*$-algebra (resp., every JB$^*$-algebra) is a JB$^*$-triple with respect to the product
$\J abc = \frac12 \ ( a b^* c + cb^* a) $ (resp., $\J abc := (a\circ b^*) \circ c + (c\circ b^*) \circ a - (a\circ c) \circ b^*$).\smallskip

For the theory of C$^*$-algebras and von Neumann algebras, we shall refer to the monographs \cite{KR86} and \cite{Tak}.  For the theory of JB$^*$-algebras and JBW$^*$-algebras we refer to \cite{HancheStor} and \cite{To}. For basic facts about abstract Jordan triple systems, consult \cite[section 1.2]{chu}.  However, the Jordan triple systems we consider in this paper are concrete, so statements about them can usually be verified directly. For example, a tripotent (defined in the next section) is nothing but a partial isometry. \smallskip

A complex \emph{JBW$^*$-triple} is a complex JB$^*$-triple which is also a dual Banach space.
The structure of JBW$^*$-triples is fairly well understood.  Every JBW$^*$-triple is a direct sum of a JBW$^*$-triple of type I and a continuous JBW$^*$-triple (defined below). JBW$^*$-triples of type I have been classified in \cite{Horn87} and continuous JBW$^*$-triples have been classified in \cite{Horn88}.
Since it is the continuous JBW$^*$-triples that concern us here, we shall not define type I, but we shall state their classification theorem from \cite{Horn87}:  A JBW$^*$-triple of type I is an $\ell^\infty$-direct sum of JBW$^*$-triples of the form $A\otimes C$, where $A$ is a commutative von Neumann algebra and $C$ is a Cartan factor (for Cartan factors, see \cite[Theorem 2.5.9 and page 168]{chu}).\smallskip

A $JBW^\ast$-triple $\mathcal{A}$ is said to be
\textit{continuous} if it has no type I direct summand. In this
case it is known that, up to isometry, $\mathcal{A}$ is a $JW^\ast$-triple, that is, a subspace of the bounded operators on a Hilbert space which is closed under the triple product $xy^*z+zy^*x$ and closed in the weak operator topology. More importantly,  it has a
unique decomposition into weak$^*$-closed ideals, $\mathcal{A} = H(W,\alpha)\oplus pV,$ where $W$
and $V$ are continuous von Neumann algebras, $p$ is a projection
in $V$, $\alpha$ is an  involution on $W$ commuting with $^*$, that is,  a $^*$-antiautomorphism of $W$ order 2, which we shall call henceforth a $\mathbb{C}$-linear
$^*$-involution,  and
$H(W,\alpha)=\{x\in W: \alpha(x)=x\}$ (see \cite[(1.20)]{Horn88}).  Notice that the triple product in $pV$ is given by $(xy^*z+zy^*x)/2$ and that $H(W,\alpha)$ is a JBW$^*$-algebra with the Jordan product $x\circ y=(xy+yx)/2$.\smallskip

 We shall show in section~\ref{section5} that for continuous JBW$^*$-triples, every 2-local triple derivation is a derivation.  (We are leaving  the study of 2-local triple derivations on the JBW$^*$-triples of type I  as one of the problems at the end of this  paper---see Problem~\ref{5.9}(a).)

\section{2-local triple derivations on right ideals of von Neumann algebras}\label{section3}
Recall that a (not necessarily linear) mapping $\Delta$ on a Jordan triple $E$ is said to be a 2-local triple derivation if, given two points $x,y\in E$, there is a triple derivation $D_{x,y}$ on $E$ such that $\Delta(x)=D_{x,y}(x)$ and $\Delta(y)=D_{x,y}(y)$.
Every 2-local triple derivation $\Delta: E\to E$ is homogeneous. Indeed, for each $a\in E$, $t\in \mathbb{C}$ consider a triple derivation $D_{a,ta}$.  Then   $\Delta(t a) = D_{a,ta} (ta) = t D_{a,ta} (a) = t \Delta(a)$.\smallskip

An element $e$ in a Jordan triple $E$ is called a \emph{tripotent}
if $\J eee =e$. Each tripotent $e$ in $E$ induces a
decomposition of $E$ (called \emph{Peirce decomposition}) in the
form: $$E=E_0(e)\oplus E_1(e)\oplus E_2(e),$$ where
$E_k(e)=\{x\in E:L(e,e)x=\frac k2 x\}$ for $k=0,1,2$ (compare \cite[page 32]{chu}). \smallskip

\begin{lemma}\label{zerotrip}
Let $\Delta : \mathcal{A} \to \mathcal{A}$ be a 2-local triple
derivation on a JB$^*$-triple. Suppose $v$ is a tripotent in
$\mathcal{A}$ such that $\Delta(v)=0$. Then
$\Delta(\mathcal{A}_k(v)) \subseteq \mathcal{A}_k(v)$, for every
$k=0,1,2$.
\end{lemma}

\begin{proof} Let $x\in \mathcal{A}_k (v)$ with $k=0,1,2$,
that is, $\{v,v,x\} = \frac{k}{2} x$. Since
\begin{eqnarray*}
\{v,v,\Delta(x)\} & = & \{v,v,D_{v,x}(x)\} = D_{v,x}
\left(\{v,v,x\}\right) -
\{D_{v,x}(v),v,x\} -\{v,D_{v,x}(v),x\}=\\
& = & D_{v,x} \left(\frac{k}{2} x\right) - \{\Delta(v),v,x\}
-\{v,\Delta(v),x\} = \frac{k}{2} D_{v,x} ( x) = \frac{k}{2}
\Delta(x).
\end{eqnarray*}
The proof is complete.
\end{proof}

We recall the following result (see \cite[Theorem
2.14]{KOPR2014}).

\begin{theorem}\label{vNa} \cite{KOPR2014}
Let $M$ be  an arbitrary von Neumann algebra and let   $T: M\to M$
be a {\rm(}not necessarily linear nor continuous{\rm)} 2-local
triple derivation. Then $T$ is a triple derivation.
\end{theorem}

Throughout this section
 $\mathcal{A}$ will denote the $JBW^\ast$-triple $pM$ where $M$ is a  von Neumann algebra and $p$ is a projection in $M$.
The following is the main result of this section. The proof will be carried out in the next subsections.

\begin{theorem}\label{triple-rectangular}
Let $M$ be a von Neumann algebra and let $p$ be  a projection in
$M.$  Then any 2-local triple derivation $\Delta$ on the
$JBW^\ast$-triple $\mathcal{A}=pM$ is  a triple derivation.
\end{theorem}

Let $a$ and $b$ be skew-hermitian elements in $pMp$ and $M,$
respectively. Let $L_a$ and $R_b$ be the left and right
multiplication operators, i.e.
\begin{equation}\label{left}
L_a(x)=ax,\, x\in \mathcal{A}
\end{equation}
and
\begin{equation}\label{right}
R_b(x)=xb,\, x\in \mathcal{A}.
\end{equation}
It is clear that $L_a$ and $R_b$ both are triple derivations on $M$, and in particular on
$\mathcal{A}.$\smallskip

Let $u$ be a tripotent in the $JBW^\ast$-triple $\mathcal{A}=pM$,
and let $(\mathcal{A}_2(u),\cdot_{u},^{*_{u}})$ denote the von Neumann
algebra whose underlying Banach space is the Pierce-2-space $\mathcal{A}_2(u)=uu^\ast
Mu^\ast u$, and whose product and involution are given by
$x\cdot_{u} y = x u^\ast y$ and $x^{*_{u}}= u x^\ast u$,
respectively.\smallskip

Let $\{.,.,.\}_1$ denote the triple product
associated to $\mathcal{A}_2(u)$, i.e. $\{x,y,z\}_1 = \frac12 (x
\cdot_{u} y^{*_{u}} \cdot_{u} z + z \cdot_{u} y^{*_{u}} \cdot_{u}
x ).$
By direct calculation,  $\{x,y, z\}_1=\{x,y,z\}.$  This also follows since the identity map is a linear isometry, and therefore an isomorphism (\cite[Proposition (5.5)]{kaup}).  Therefore a
linear mapping $D: \mathcal{A}_2(u) \to \mathcal{A}_2(u)$ is a
triple derivation (resp. 2-local triple derivation) for the product $\{.,.,.\}$ if and only if it is
a triple derivation (resp. 2-local triple derivation)  for the product $\{.,.,.\}_1$.

\subsection{Properly infinite case}

In this subsection we will consider 2-local triple derivations on
$JBW^\ast$-triples of the form $\mathcal{A}=pM,$ where $p$ is a
properly  infinite projection in a  von Neumann algebra $M.$\smallskip

Let $q$ be a projection in $M$ and let  $D$ be a
triple derivation on $\mathcal{A}=pM.$ It is easily seen that
an operator $D_{(q)}$ on the  $JBW^\ast$-subtriple $pMq$ defined
by
\[
D_{(q)}(x)=D(x)q,\, x\in pMq
\]
is a  triple derivation on $pMq.$
Thus, if  $\Delta$ is a 2-local
triple derivation on $\mathcal{A}=pM,$ then the  operator $\Delta_{(q)}$ on the $JBW^\ast$-subtriple $pMq$ defined
by
\begin{equation}\label{resq}
\Delta_{(q)}(x)=\Delta(x)q,\, x\in pMq
\end{equation}
is a 2-local triple derivation on $pMq.$\smallskip

The following is the main result of this subsection.

\begin{theorem}\label{prop}
Let $M$ be a  von Neumann algebra and let $p$ be a properly
infinite projection in $M.$  Then any 2-local triple derivation
$\Delta$ on $\mathcal{A}=pM$ is  a triple derivation.
\end{theorem}

\begin{proof}  Since $p$ is
properly infinite, by  using the  halving Lemma  five times (see for example \cite[Lemma 6.3.3]{KR86})  we
can find mutually orthogonal projections $e_1, \ldots, e_6$ in $M$
such that  $p\sim e_1\sim\ldots \sim e_6$ and $p=e_1+\cdots +e_6.$\smallskip

Denote by $r(x)$ and $l(x)$ the right and left supports in $M$ of
the element $x$ from $M,$ respectively. Since $r(x)\sim l(x)$ (see \cite[Proposition V.1.5]{Tak})  and
$l(x)\leq p,$ it follows that $r(x)\preceq p$ for all $x\in
\mathcal{A}.$\smallskip

Let $x,y\in \mathcal{A}.$ Denote by $q_1, \ldots, q_6$ the right
supports of elements $x,$ $y,$ $x+y,$ $\Delta(x),$ $\Delta(y)$ and
$\Delta(x+y),$ respectively. Then $q_i\preceq p$ for all $i\in
\{1,\ldots,6\}.$ Since $p\sim e_i$ for all $i,$
it follows that  $q_i\preceq e_i$ for all $i\in \{1,\ldots,6\}.$
Therefore $\bigvee\limits_{i=1}^6q_i\preceq e_1+\ldots +e_6=p$
(see \cite[Exercise 6.9.3]{KR92}).\smallskip

Let us  show the existence of a  projection $q\in M$ such that
$\bigvee\limits_{i=1}^6 q_i\leq q\sim p=e_1+\cdots +e_6.$ Since
$p$ is properly infinite by \cite[Exercise 6.9.4]{KR92} it follows
that
$$
\left(\bigvee\limits_{i=1}^6 q_i\right)\vee p\sim p.
$$
Then it  suffices to take $q=\left(\bigvee\limits_{i=1}^6
q_i\right)\vee p.$\smallskip

Since $p\sim q$ there exists a partially isometry $u\in M$ such
that $uu^\ast =p,$ $u^\ast u=q.$  As was mentioned before this
subsection, $p M q=uu^\ast M u^\ast u$ is a von Neumann algebra
with respect to product and involution given by $x\cdot_{u} y = x
u^\ast y$ and $x^{*_{u}}= u x^\ast u$, respectively. \smallskip

Let $\Delta_{(q)}$ be the 2-local triple derivation on $pMq$ defined
by \eqref{resq}. Then by Theorem~\ref{vNa},  $\Delta_{(q)}$ is a
triple derivation. By the construction of $q$ it follows that $x,
y, x+y, \Delta(x), \Delta(y), \Delta(x+y)$ all belong to
$pMq.$ Therefore
\begin{eqnarray*}
\Delta(x+y) & = & \Delta(x+y)q=\Delta_{(q)}(x+y)=
\Delta_{(q)}(x)+\Delta_{(q)}(y)= \\
& =& \Delta(x)q+\Delta(y)q=\Delta(x)+\Delta(y).
\end{eqnarray*}
Thus $\Delta$ is additive and hence linear.  Since every (linear) local triple derivation on a JB*-triple is automatically continuous and hence a triple derivation (see \cite[Theorem 2.8]{BurPolPerBLMS14}),  the proof is complete.
\end{proof}

\subsection{Finite case}

In this subsection we will consider 2-local triple derivations on
$JBW^\ast$-triples of the form $\mathcal{A}=pM,$ where $p$ is a
finite projection in a  von Neumann algebra $M.$\smallskip

Let $D$ be a triple derivation on $\mathcal{A}.$ Set, for a tripotent $u\in \mathcal{A}$,
\begin{equation}\label{res}
D^{(u)}(x)=\{u,\{u,D(x),u\},u\}=uu^\ast D(x)u^\ast u,\, x\in
\mathcal{A}_2(u).
\end{equation}
It is easily seen that $D^{(u)}$ is a triple derivation on
$\mathcal{A}_2(u).$\smallskip

Let $\Delta$ be a 2-local triple derivation on $\mathcal{A}$ and
let $u$ be a tripotent in $\mathcal{A}.$ Then
\begin{equation}\label{tripo}
uu^\ast \Delta(u)u^\ast u=-u \Delta(u)^\ast u.
\end{equation}

Indeed, take a triple derivation $D$ on $\mathcal{A}$ with
$\Delta(u)=D(u).$ From the equality $\{u, u, u\}=u,$ we have that
\begin{equation}\label{trid}
uu^\ast D(u)u^\ast u=-u D(u)^\ast u,
\end{equation}
which implies \eqref{tripo}.

\begin{lemma}\label{aabb}
Let $\Delta$ be a 2-local derivation on $\mathcal{A}.$
There exist skew-hermitian elements $a_1$ in $pMp$ and $b_1$ in
$M$ such that $$ \Delta(p)=L_{a_1}(p)+R_{b_1}(p).
$$
\end{lemma}

\begin{proof}
Set
\begin{center}
$a_1=\Delta(p)p$ and $b_1=\Delta(p)p^{\perp}-p^{\perp}\Delta(p)^\ast,$
\end{center}
where $p^{\perp}=\mathbf{1}-p.$ From~\eqref{tripo} it follows that  $a_1$
is skew-hermitian. It is clear that $b_1$ is also skew-hermitian.
We have
$$
L_{a_1}(p)+R_{b_1}(p)=a_1p+pb_1=\Delta(p)p+p\Delta(p)p^{\perp}=\Delta(p)p+\Delta(p)p^{\perp}=\Delta(p).
$$
\end{proof}

\begin{lemma}\label{restr}
Let $\Delta$ be a 2-local derivation on $\mathcal{A}.$
Suppose that $\Delta(p)=0.$ Then there exists a skew-hermitian
element $a_2$ in $pMp$ such that $\Delta(x)=L_{a_2}(x)-R_{a_2}(x)$
for all $x\in \mathcal{A}_2(p)=pMp.$
\end{lemma}

\begin{proof} Since $\Delta(p)=0,$ Lemma~\ref{zerotrip} implies
that $\Delta$ maps  $\mathcal{A}_2(p)=pMp$ into itself.\smallskip

Let $x, y\in \mathcal{A}_2(p).$ Take a triple derivation $D_{x,y}$ on
$\mathcal{A}$ such that
$$
\Delta(x)=D_{x,y}(x),\, \Delta(y)=D_{x,y}(y).
$$
Let $D_{x,y}^{(p)}$ be the triple derivation defined by \eqref{res}. Then
$$
\Delta(x)=D_{x,y}^{(p)}(x),\, \Delta(y)=D_{x,y}^{(p)}(y).
$$
This means that the restriction $\Delta|_{\mathcal{A}_2(p)}$ is a
2-local triple  derivation on the von Neumann algebra
$\mathcal{A}_2(p).$ By Theorem~\ref{vNa},
$\Delta|_{\mathcal{A}_2(p)},$ is a triple derivation. Since
$\Delta(p)=0,$ there exists a skew-hermitian element $a_2$ in
$pMp$ such that $\Delta(x)=a_2x-xa_2$ for all
$x\in \mathcal{A}_2(p)=pMp$ (see \cite[beginning of section 2]{KOPR2014}).
\end{proof}

Let $D$ be an arbitrary triple derivation (or a 2-local triple derivation) on $\mathcal{A}$. Then
$D$ can be decomposed in the form
\begin{equation}\label{deco}
D=D_1+D_2,
\end{equation}
where $D_1=L_a+R_b$, with $a,b$ skew-hermitian and $D_2|_{\mathcal{A}_2(p)}\equiv0.$\smallskip

Indeed, by Lemma~\ref{aabb},  there exist skew-hermitian elements $a_1\in pMp$ and $b_1\in
M$ such that $(D-L_{a_1}-R_{b_1})(p)=0$.  By Lemma~\ref{restr},  there is a skew-hermitian element $a_2\in pMp$ such that $(D-L_{a_1}-R_{b_1})(x)=L_{a_2}x-R_{a_2}x$ for all $x\in pMp$.
Now it suffices to set
\begin{center}
$D_1=L_{a_1+a_2}+R_{b_1-a_2}$ and $D_2=D-D_1.$
\end{center}

\begin{lemma}\label{skewone}
Let $D$ be a triple derivation on $\mathcal{A}$ such that
$D|_{\mathcal{A}_2(p)}\equiv0.$ Then
\begin{equation}\label{sk}
D(x)y^\ast+xD(y)^\ast=0
\end{equation}
for all $x,y\in \mathcal{A}.$
\end{lemma}

\begin{proof} Let us first consider  a case $x, y\in \mathcal{A}_1(p)=pM(\mathbf{1}-p).$\smallskip

Since $D|_{\mathcal{A}_2(p)}\equiv0,$ it follows from Lemma~\ref{zerotrip} that $D$ maps
$\mathcal{A}$ into $\mathcal{A}_1(p).$  Taking into account these
properties we have
\begin{eqnarray*}
x D(y)^\ast  & = &    x D(y)^\ast p+p D(y)^\ast x= 2\{x, D(y), p\}=\\
& = & 2D\left(\{x,y,p\}\right)-2\{D(x), y, p\}-2\{x, y, D(p)\}=\\
& = & D(xy^\ast p+py^\ast x)-D(x)y^\ast p- py^\ast D(x)=\\
& = & D(xy^\ast)-D(x)y^\ast=-D(x)y^\ast,
\end{eqnarray*}
i.e. $D(x)y^\ast+xD(y)^\ast=0$ for $x,y\in{\mathcal A}_1(p).$\smallskip

Let now $x, y\in \mathcal{A}$ be arbitrary and let $x=x_2+x_1,
y=y_2+y_1\in \mathcal{A}=\mathcal{A}_2(p)+\mathcal{A}_1(p).$ We
have
\begin{eqnarray*}
D(x)y^\ast+xD(y)^\ast & = &
D(x_2+x_1)(y_2+y_1)^\ast+(x_2+x_1)D(y_2+y_1)^\ast= \\
& = & D(x_1)y_2^\ast+x_1D(y_1)^\ast+D(x_1)y_1^\ast+x_2D(y_1)^\ast\\
& = & D(x_1)y_2^\ast+x_2D(y_1)^\ast=0,
\end{eqnarray*}
because $D(x_1)y_2^\ast=(D(x_1)(\mathbf{1}-p))(y_2p)^\ast=0$ and
$x_2D(y_1)^\ast=x_2p(D(y_1)(\mathbf{1}-p))^\ast=0.$ The proof is
complete.
\end{proof}

Since $pMp$ is finite, there exists a faithful center-valued trace
$\tau$ on $pMp,$ that is, a linear map from $pMp$ into the center,
$Z(pMp)$,  of $pMp$ such that
\begin{enumerate}
\item[(i)]  $\tau(xy)=\tau(yx)$ for all $x,y\in pMp;$
\item[(ii)] $\tau(z)=z$ for all $z\in Z(pMp);$
\item[(iii)] $\tau(xx^\ast)=0$ implies $x=0.$
\end{enumerate}

Define a $Z(pMp)$-valued sesquilinear form on $\mathcal{A}$ by
$$
\langle x,y\rangle=\tau(xy^\ast),\, x,y\in \mathcal{A}.
$$

Since $\tau$ is faithful it follows that the form
$\langle\cdot,\cdot\rangle$ is non-degenerate, i.e. $\langle
x,y\rangle=0$ for all $y\in \mathcal{A}$ implies that $x=0.$

\begin{lemma}\label{ccc}
Let $D$ be an arbitrary triple derivation on $\mathcal{A}.$ Then
$$
\langle D(x), y\rangle = -\langle x, D(y)\rangle
$$
for all $x, y\in \mathcal{A}.$
\end{lemma}

\begin{proof} Let $D=D_1+D_2$ be a decomposition of $D$ in the
form \eqref{deco}. For $x,y\in {\mathcal A}$, we have $a^*=-a\in pMp$ and $b^*=-b\in M$ such that
\begin{eqnarray*}
D_1(x) y^\ast +x D_1(y)^\ast  & = & (ax+xb)y^\ast + x(ay+yb)^\ast =\\
& = & axy^\ast +xby^\ast +xy^\ast a^\ast+xb^\ast y^\ast=\\
& = & axy^\ast +xby^\ast -xy^\ast a-x b y^\ast=\\
& = & axy^\ast -xy^\ast a,
\end{eqnarray*}
i.e.
\begin{eqnarray*}
D_1(x) y^\ast +x D_1(y)^\ast  & = & axy^\ast -xy^\ast a.
\end{eqnarray*}
Since  a center-valued trace annihilates commutators we have
that
\begin{eqnarray*}
\tau\left(D_1(x) y^\ast +x D_1(y)^\ast\right)=0.
\end{eqnarray*}
Thus
$$
\langle D_1(x), y\rangle = -\langle x, D_1(y)\rangle.
$$

On the other hand, by Lemma~\ref{skewone} it follows that
$$
\langle D_2(x), y\rangle = -\langle x, D_2(y)\rangle.
$$
The proof is complete.
\end{proof}

The following is the main result of this subsection.

\begin{theorem}\label{finite}
Let $M$ be a  von Neumann algebra and let $p$ be a finite
projection in $M.$  Then any 2-local triple derivation $\Delta$ on
$\mathcal{A}=pM$ is  a triple derivation.
\end{theorem}

\begin{proof}
Let us first  show that
$$
\langle \Delta(x), y\rangle = -\langle x, \Delta(y)\rangle
$$
for all $x, y\in \mathcal{A}.$\smallskip

Take a triple derivation $D$ on $\mathcal{A}$ such that
\begin{center}
$\Delta(x)=D(x)$ and $\Delta(y)=D(y).$
\end{center}
By Lemma~\ref{ccc}, we have
$$
\langle \Delta(x), y\rangle = \langle D(x), y\rangle = -\langle x,
D(y)\rangle=-\langle x, \Delta(y)\rangle.
$$

Let now $x,y,z$ be arbitrary elements in $\mathcal{A}.$ Then
\begin{eqnarray*}
\langle \Delta(x+y), z\rangle  & = & -\langle x+y,
\Delta(z)\rangle =-
\langle x, \Delta(z)\rangle-\langle y, \Delta(z)\rangle=\\
& = & \langle \Delta(x), z\rangle+\langle \Delta(y), z\rangle=
\langle \Delta(x)+\Delta(y), z\rangle,
\end{eqnarray*}
i.e.
\begin{eqnarray*}
\langle \Delta(x+y)-\Delta(x)-\Delta(y), z\rangle=0.
\end{eqnarray*}
Since $z$ is an arbitrary and the sesquilinear  form is
non-degenerate it follows that $\Delta(x+y)=\Delta(x)+\Delta(y)$, so $\Delta$ is additive, hence linear, hence a triple derivation by \cite[Theorem 2.8]{BurPolPerBLMS14} (compare the proof of Theorem~\ref{prop}).
\end{proof}

\subsection{General case}

We need the following two Lemmata.

\begin{lemma}\label{module-homo}
Let $D$  be  a triple derivation on $pM.$ Then $D$ is
$\mathcal{P}(Z(M))$-homogeneous, i.e.
$$
D(cx)=cD(x)
$$
for any central projection  $c\in \mathcal{P}(Z(M))$ and $x\in pM.$
\end{lemma}

\begin{proof}  Let $c\in \mathcal{P}(Z(M)).$
Take $x,y,z\in pM.$  We have
\begin{eqnarray*}
c\{ x, D(c y),z\}  & = & c D\left(\{x, c y, z\}\right)-c \{D(x),
c y, z\}-c \{x, c y, D(z)\} =\\
& = & c D\left(\{x, c y,z\}\right)-c \{c D(x), y, z\}-c \{x, y,
D(z)\}
\end{eqnarray*}
and
\begin{eqnarray*}
c\{x, D(c y), z\}  & = & c \{c x, D(c y), z\} = \\
& = & c D\left(\{c x, c
y,z\}\right)-c \{D(c x),c y,z\}-c \{c x, c y, D(z)\} =\\
& = & c D\left(\{x, c y, z\}\right)-c \{D(c x), y, z\}-c \{x, y,
D(z)\}.
\end{eqnarray*}
Thus $ c\{c D(x), y, z\}=c\{D(c x), y, z\}.$ Since $c$ is a
central projection we obtain  that
$$
\{c D(x), y, z\}=\{c D(c x), y, z\}.
$$
Since $y,z$ are arbitrary, it follows that
\begin{equation}\label{cc}
cD(x)=cD(c x).
\end{equation}
Thus
$$
cD((\mathbf{1}-c) x)=0.
$$

Replacing $c$ by $\mathbf{1}-c$ in the last equality we obtain
that
\begin{equation}\label{eq ccc}
(\mathbf{1}-c)D(c x)=0.
\end{equation}
Thus
\begin{eqnarray*}
D(c x)  & = & (c+(\mathbf{1}-c))D(c x)=cD(c x)+(\mathbf{1}-c)D(c
x)\stackrel{\eqref{eq ccc}}=cD(c x)\stackrel{\eqref{cc}}=c D(x).
\end{eqnarray*}
The proof is complete.
\end{proof}

\begin{lemma}\label{module-homog}
Let $\Delta$  be  a 2-local triple derivation on $pM.$ Then
$\Delta$ is $\mathcal{P}(Z(M))$-homogeneous, i.e.
$$
\Delta (cx)=c\Delta (x)
$$
for any central projection  $c\in \mathcal{P}(Z(M))$ and $x\in pM.$
\end{lemma}

\begin{proof}  Let $c\in \mathcal{P}(Z(M))$ and  $x\in pM.$ Let $D_{cx,x}: pM \to pM$ be a triple derivation satisfying $\Delta (cx) = D_{cx,x} (cx)$ and $\Delta (x) = D_{cx,x} (x)$.  By Lemma~\ref{module-homo}, we have
\begin{eqnarray*}
\Delta(c x)  & = &  D_{cx, x}(c x
)= c D_{cx, x}(x )=c \Delta(x).
\end{eqnarray*}
\end{proof}

Now we are in position to prove Theorem~\ref{triple-rectangular}.

\begin{proof}[\textit{Proof of Theorem~\ref{triple-rectangular}}]

Let $M$ be a von Neumann algebra, $p$ be  a projection in $M$ and
$\Delta$ be a 2-local triple derivation  on the $JBW^\ast$-triple
$\mathcal{A}=pM.$\smallskip

Take mutually orthogonal central projections $z_1$  and $z_2$ in
$M$ with  $z_1+z_2=\mathbf{1}$  such that  $z_1p$ is finite and
$z_2p$ is properly infinite. Lemma~\ref{module-homog} implies that
$\Delta$ maps each $z_i\mathcal{A}$ into itself and hence induces
a 2-local triple derivation  $\Delta_i=\Delta|_{z_i \mathcal{A}}$
on $z_i\mathcal{A}=z_ipM$  for $i=1,2.$
Theorems~\ref{prop},~\ref{finite} imply that both $\Delta_1$ and
$\Delta_2$ are triple derivations. Since
$$
\Delta(x)=z_1\Delta(x)+z_2\Delta(x)=\Delta_1(z_1 x)+\Delta_2(z_2
x)
$$
for all $x\in \mathcal{A},$ it follows that $\Delta$ is also a
triple derivation. The proof is complete.
\end{proof}

A Cartan factor of type 1 is the JBW$^*$-triple $B(H,K)$ of all bounded operators from a Hilbert space $H$ to a Hilbert space $K$.  We thus have:

\begin{corollary}\label{Cartan}
Every 2-local triple derivation on a Cartan factor of type 1 is a triple derivation.
\end{corollary}

\section{Boundedness of completely additive measures on continuous JW*-algebras}\label{susect: Dorofeev}

In this section we shall establish one of the  main results of this note, namely a Jordan version of  Dorofeev's boundedness theorem (compare \cite[Theorem 29.5]{Shers2008} or \cite[Theorem 1]{Doro}). The latter states that any completely additive signed measure on the projections of a continuous von Neumann algebra is bounded. \smallskip

 Theorem~\ref{t Jordan Dorofeev-Sherstnev} provides the key tool for the proof of Theorem~\ref{t continuous type 2}, which together with Theorem~\ref{triple-rectangular} leads to the second main conclusion of this note in Theorem~\ref{continuous}, namely, that a 2-local triple derivation on a continuous JBW*-triple is a triple derivation.\smallskip

Assume that $M$ is a continuous
von Neumann algebra and $\beta: M\to M$ is a $\mathbb{C}$-linear
$^*$-involution (i.e. a $^*$-antiautomorphism of order 2). The subspace $H(M,\beta)$, of all $\beta$-fixed
points in $M$, is not, in general, a subalgebra of $M$. However,
$H(M,\beta)$ is a weak$^*$ closed Jordan $^*$-subalgebra of $M$,
whenever the latter is equipped with its natural Jordan product
$$\displaystyle x\circ y := \frac12 (x y + y x).$$ In particular,
the self-adjoint part, $H(M,\beta)_{sa}$, of $H(M,\beta)$ is a
JBW-subalgebra of $M_{sa}$.

\begin{theorem}\label{t Jordan Dorofeev-Sherstnev} Let $M$ be a continuous von Neumann algebra and let $\beta: M\to M$ be a $\mathbb{C}$-linear $^*$-involution.
Let $\Delta: \mathcal{P}(H(M,\beta))\to \mathbb{C}$ be a completely additive (complex) measure. Then $\Delta$ is bounded.
\end{theorem}

 The authors do not know if Theorem~\ref{t Jordan Dorofeev-Sherstnev} remains valid when $H(M,\beta)$ is replaced by an arbitrary JBW$^*$-algebra containing no summands of type $I_n$. See Problem~\ref{5.10}.  However,
Theorem~\ref{t Jordan Dorofeev-Sherstnev} is sufficient for the purposes of this paper.\smallskip

We shall show how the arguments in \cite{Doro} can be adapted to prove the above result. For completeness reasons, we shall present here a draft of the original arguments employed in the proof of \cite[Theorem 1]{Doro}, making the adjustments, some of which are non-trivial, for the Jordan case.
The proof of Theorem~\ref{t Jordan Dorofeev-Sherstnev} will occupy us throughout this section.\smallskip

The following Jordan version of the Bunce-Wright-Mackey-Gleason theorem is an instance of a theorem due to Matve\u{i}chuk and has been borrowed from \cite{Matv}

\begin{theorem}\label{t Jordan Mackey-Gleason}{\rm \cite[Theorem 1]{Matv}} Let $M$ be a continuous von Neumann algebra and let $\beta: M\to M$ be a $\mathbb{C}$-linear $^*$-involution. Let $\Delta: \mathcal{P}(H(M,\beta))\to \mathbb{C}$ be a bounded finitely additive measure. Then there exists a functional $\varphi$ in $H(M,\beta)^*$ such that $\Delta (p) =\varphi(p)$, for every $p\in \mathcal{P}(H(M,\beta))$. Furthermore, when $\Delta$ is completely additive the functional $\varphi$ can be assumed to be in $H(M,\beta)_*$. $\hfill\Box$
\end{theorem}

Suppose that $M$ acts on a complex Hilbert space $H$. Following \cite{Doro}, given two projections $p,q\in \mathcal{P} (M)$, the distance between $p$ and $q$ is defined by $$d(p,q) = \inf\{\|\xi-\eta\|: \xi\in p(H), \eta\in q(H), \|\xi \|= \|\eta\|=1\}.$$ Let us take $\xi\in p(H), \eta\in q(H)$ with $\|\xi \|= \|\eta\|=1$.   In this case $$\|\xi-\eta \| \geq \|\xi-q(\xi) \|-\|q(\xi-\eta)\| \geq \|\xi-q(\xi) \|-\|\xi-\eta\|,$$ which gives $2 \|\xi-\eta\| \geq \|\xi-q(\xi) \|$. Therefore $$2 \|\xi-\eta\| \geq \inf \{ \|\zeta-q(\zeta) \| : \zeta\in p(H), \|\zeta \|=1\},$$ and thus \begin{equation}\label{eq inque distance projections} d(p,q) \geq \frac12 \inf \{ \|\zeta-q(\zeta) \| : \zeta\in p(H), \|\zeta \|=1\}.
\end{equation}

Following standard notation, given two projections $p,q$ in a von Neumann algebra $M$, the symbols $p\vee q$ and $p\wedge q$ will denote the supremum and the infimum of $p$ and $q$ in $M$, respectively. Let $\beta$ be a $\mathbb{C}$-linear $^*$-involution on $M$. It is clear that $\beta(\mathbf{1}) =\mathbf{1}$. Furthermore, $\beta (p\vee q) = \beta (p)\vee \beta(q)$ and $\beta(p\wedge q)= \beta(p)\wedge \beta(q)$. So, if $p,q\in H(M,\beta)$, then $p\vee q$ and $p\wedge q$ both belong to $H(M,\beta)$. Having these comments in mind, the arguments in the proof of \cite[Lemma 2]{Doro} can be slightly adapted  to obtain:

\begin{lemma}\label{l 2 Dorofeev in Jordan} Let $M$ be a continuous von Neumann algebra and let $\beta: M\to M$ be a $\mathbb{C}$-linear $^*$-involution.
Let $\Delta: \mathcal{P}(H(M,\beta))\to \mathbb{C}$ be a completely additive measure. Suppose there exists a constant $C>0$ and an increasing sequence $(q_n)$ of projections in $H(M,\beta)$ such that $(q_n)\uparrow \mathbf{1}$ and  $$\sup \{ |\Delta(q)| : q\in \mathcal{P}(H(M,\beta)) : q \leq q_n\} \leq C,$$ for every natural $n$. Then $\Delta$ is bounded.
\end{lemma}

\begin{proof} Let us observe that $\Delta$ being a completely additive  measure implies that for every increasing (respectively, decreasing) sequence $(r_n)$ in $\mathcal{P}(H(M,\beta))$ with $(r_n)\uparrow r$ (respectively, $(r_n)\downarrow r$), where $r\in \mathcal{P}(H(M,\beta))$, then $\Delta (r_n) \to \Delta (r)$.\smallskip

We shall show that the set $\{|\Delta(p)| : p\in \mathcal{P}(H(M,\beta))\}$ is bounded. Let us fix $p\in \mathcal{P}(H(M,\beta))$. Since $(p_n) = (\mathbf{1}-q_n) \downarrow 0$ and $p_n\wedge (\mathbf{1}-p) \leq p_n$, we deduce that $(|\Delta(p_n)|)$, and $(|\Delta(p_n\wedge (\mathbf{1}-p))|)$ tend to $0$. We can therefore assume that $|\Delta(p_n)|, |\Delta(p_n\wedge (\mathbf{1}-p))| \leq 1$, for each natural $n$.\smallskip

We claim that for each natural $n$ and every projection $r\in \mathcal{P}(H(M,\beta))$ we have \begin{equation}\label{eq maximum bounded} |\Delta (r\vee p_n)| \leq 1+C.
\end{equation} Indeed, since $r\vee p_n\geq p_n$, it follows that $r\vee p_n = r\vee p_n -p_n + p_n$ with $r\vee p_n-p_n \perp p_n$. Therefore $\Delta(r\vee p_n) = \Delta(r\vee p_n -p_n)  + \Delta(p_n).$ Since $p_n (r\vee p_n -p_n) = 0 = (r\vee p_n -p_n) p_n$, we deduce that $r\vee p_n -p_n \leq \mathbf{1}- p_n = q_n$. It follows from the assumptions that $|\Delta(r\vee p_n)| \leq |\Delta(r\vee p_n -p_n)|  + |\Delta(p_n)|\leq C+1,$ as desired.\smallskip

With $p$ as above, let us denote $q= p + (\mathbf{1}-p) \wedge p_1$. It is easy to check that $(\mathbf{1}-q) \wedge p_1=0$, therefore Remark 1 in \cite{Doro} proves that $q = r (q p_1 q) + q\wedge (\mathbf{1}-p_1)$ with $r (q p_1 q) \perp q\wedge (\mathbf{1}-p_1)$. Since $p \perp (\mathbf{1}-p) \wedge p_1$, we deduce from the finite additivity of $\Delta$ that $\Delta (q) = \Delta(p) + \Delta((\mathbf{1}-p) \wedge p_1)$ and $\Delta(q)= \Delta (r (q p_1 q)) +\Delta( q\wedge (\mathbf{1}-p_1))$, and hence $$|\Delta (p) | \leq  |\Delta (q) | + |\Delta((\mathbf{1}-p) \wedge p_1)| \leq |\Delta (r (q p_1 q))| + |\Delta (q\wedge (\mathbf{1}-p_1)) |  + 1$$ $$\leq |\Delta (r (q p_1 q))| + C   + 1.$$

The sequence $(G_n)= (1_{(0,1-\frac1n)} (q p_1 q))\subseteq  \mathcal{P}(H(M,\beta))\}$ grows to the range projection $r (q p_1 q)$. We deduce that $(\Delta (G_n))\uparrow \Delta (r (q p_1 q))$, and thus, there exists $n_1\in \mathbb{N}$ such that $|\Delta(G_{n_1}) - \Delta (r (q p_1 q))|<1,$ and consequently, \begin{equation}\label{eq inequ G and range} |\Delta (p) | \leq 2+C + |\Delta(G_{n_1})|.
\end{equation}

We claim that $G= G_{n_1}$ is ``separated'' from $p_1$ in the sense of \cite{Doro}, that is, $d(p_1,G_{n_1})>0$. Considering the von Neumann subalgebra generated by the element $qp_1 q$ and the functional calculus it is easy to see that $qp_1 q \leq (1-\frac{1}{n_1}) G + 1_{[1-\frac{1}{n_1},1]} (q p_1 q)$, with $(1-\frac{1}{n_1}) G \perp 1_{[1-\frac{1}{n_1},1]} (q p_1 q)$. Then for each normal state $\varphi\in M_*$ with $\varphi (G) = 1 = \|\varphi\|$, we have $\varphi (q p_1 q)\leq 1-\frac{1}{n_1}$. Consequently, for each $\xi\in G(H)$ with $\|\xi\|=1$, we have $\langle qp_1 q (\xi) / \xi \rangle \leq 1-\frac{1}{n_1}.$ Having in mind that $G\leq r(qp_1 q) \leq q$, we deduce that $q(\xi) = \xi$, and hence $\langle p_1  (\xi) / \xi \rangle \leq 1-\frac{1}{n_1},$ for every $\xi$ as above. This shows that $\|\xi -p_1 (\xi) \|\geq \frac{1}{\sqrt{n_1}}$, for every $\xi$ satisfying the above conditions. The inequality in \eqref{eq inque distance projections} shows that $d(G,p_1)\geq \frac{1}{2 \sqrt{n_1}}$, proving that $G$ is separated from $p_1$. Therefore, $d(G,p_n)\geq \frac{1}{2 \sqrt{n_1}}$, for every $n\in \mathbb{N}$. Lemma 1$(b)$ in \cite{Doro} shows that $$G\vee p_n \leq \frac{16}{d(G,p_n)^2} (G+p_n)\leq 64 n_1 (G+p_n),$$ for every natural $n$. We deduce that $\displaystyle \lim_{n} G\vee p_n \leq \lim_{n}  64 n_1 (G+p_n)= 64 n_1 G$, which implies that $G\vee p_n\downarrow G$. We can find $n_2\in \mathbb{N}$ satisfying $|\Delta(G)|\leq 1+ |\Delta(G\vee p_{n_2})|$. Combining \eqref{eq inequ G and range} and \eqref{eq maximum bounded} we obtain $$|\Delta(p) | \leq 3+ C + |\Delta(G\vee p_{n_2})| \leq 4+ 2 C.$$ The conclusion of the lemma follows from the arbitrariness of $p$.
\end{proof}

The following result for projections in von Neumann algebras is part of the folklore (cf. \cite[Lemma 3]{Doro} or \cite[Lemma 6.1.10]{Ham03}. Let us observe that in the latter results the normal state should have been assumed to be faithful).  By using the halving lemma for JBW-algebras the same proof holds in the case of JBW$^*$-algebras.

\begin{lemma}\label{l 3 Jordan} Let $M$ be a continuous von Neumann algebra and let $\beta: M\to M$ be a $\mathbb{C}$-linear $^*$-involution. Suppose $p$ is a projection in $H(M,\beta),$ $\varphi$ is a faithful normal state in $H(M,\beta)_*$ and $0<\delta<1$. Then there exists a family of pairwise orthogonal projections $(p_i)_{i=1,\ldots,n}$ in $H(M,\beta)$ satisfying:\begin{enumerate}[$(a)$] \item $p= \sum_{i=1}^{n} p_i;$
\item $\varphi (p_i) \leq \delta,$ for every $i=1,\ldots, n;$
\item $n\leq 2/\delta.$ $\hfill\Box$
\end{enumerate}
\end{lemma}

The following result is a crucial point in the proof of the main theorem of this section.

\begin{proposition}\label{conjecture subalgebra of type II1} Let $W$ be a JW-algebra containing no finite Type I part. Then $W$ contains a JW-subalgebra $B$ of Type II$_1$. Furthermore, if $M$ is a (properly infinite) continuous von Neumann algebra and $\beta: M\to M$ is a $\mathbb{C}$-linear $^*$-involution, then there exists a type II$_1$ von Neumann subalgebra $N$ of $M$ satisfying $\beta(N)=N$.
\end{proposition}

The proof of the above proposition will follow from a technical lemma. First, we recall that a real flip $\alpha$ on $B(H)$ is a $^\ast$-antiautomorphism of order 2 given by $$ \alpha(x)= J x^\ast J\,,$$ where $J$ is a conjugation on $H$. In this setting $$B(H)^{\alpha}_{sa} = \left\{ x\in B(H) : \alpha (x)=x=x^\ast\right\}$$ is a Type I JW-algebra factor. Since any two conjugations on the same complex Hilbert space are unitarily equivalent (see \cite[Lemma 7.5.6]{HancheStor}) all factor JW-algebras arising from a real flip on a fixed Hilbert space are isomorphic.

\begin{lemma}\label{l real flip}
Let $\alpha$ be a real flip on $B(H)$, where $H$ is a separable and infinite dimensional complex Hilbert space.
Then there exists a factor von Neumann algebra $N$ of type II$_1$, such that $N$ is an $\alpha$-invariant subalgebra of $B(H)$. In particular, $H(N,\alpha)_{sa}=\{x\in N : \alpha (x)=x^*=x \}$ is a finite Type II$_1$ JW-factor contained in $B(H)_{sa}^\alpha$. Moreover, $H(N,\alpha)_{sa}$ is not isomorphic to the self-adjoint part of a von Neumann algebra and the enveloping von Neumann algebra of $H(N,\alpha)_{sa}$ coincides with $N$.
\end{lemma}

\begin{proof}
Let $\Pi$ be the group of all permutations of natural numbers leaving all but finite integers fixed. $\Pi$ is infinite and countable and so we can suppose that $H=\ell^2(\Pi)$. Denote by $\xi_t$ an element in $\ell^2(\Pi)$ that takes value 1 at $t\in \Pi$ and zero otherwise. Then $(\xi_t)_{t\in \Pi}$ forms an orthonormal basis of $H$. By the remark preceding this lemma, there is no loss of generality in assuming that the real flip $\alpha$ is induced by a conjugation $J$ of the form:
\[ J\left(\sum_{t\in \Pi} \alpha_t \xi_t\right)= \sum_{t\in \Pi} \overline{\alpha_t} \xi_t\,,\]
where $(\alpha_t)\in \ell^2(\Pi)$.
Let $\mathcal{L}_G$ be the (left) group von Neumann algebra generated by the unitaries  $u_t$  ($t\in \Pi$), where
\[ u_t \xi_s := \xi_{ts}\,.\]
Since
$$
\alpha(u_t)\xi_s=Ju^*_tJ\xi_s=Ju^*_t\xi_s=Ju_{t^{-1}}\xi_s=\xi_{t^{-1}s}=u_{t^{-1}}\xi_s=u_t^*\xi_s,
$$
 \[ \alpha(u_t)=u_t^\ast\] and consequently, $\mathcal{L}_G$ is $\alpha$-invariant, and hence $\alpha(\mathcal{L}_G)=\mathcal{L}_G$.

By \cite[Example 6.7.7, page 438 and Theorem 6.7.5]{KR86} $\mathcal{L}_G$ is a Type II$_1$ factor (see also \cite[Theorem 6.7.2]{KR86}).

Now, since $\mathcal{L}_G$ is a continuous von Neumann factor, we conclude, by Theorem 1.5.2 in \cite{AyuRakUsma97}, that the algebra $H(\mathcal{L}_G,\alpha)_{sa} :=\{x\in \mathcal{L}_G : \alpha (x)=x=x^\ast \}$ is  a continuous JW-algebra factor which is not isomorphic to the self-adjoint part of a von Neumann algebra and the enveloping von Neumann algebra of $H(\mathcal{L}_G,\alpha)_{sa}$ coincides with $\mathcal{L}_G$.
Moreover, Theorem 1.3.2 in \cite{AyuRakUsma97} implies that $H(\mathcal{L}_G,\alpha)$ is finite. 
\end{proof}

\begin{proof}[Proof of Proposition \ref{conjecture subalgebra of type II1}]
Let us suppose first that $W$ is infinite and homogeneous Type I$_n$, where $n$ is an infinite cardinal number. Then, according to the structure theory (see \cite[Definition 5.3.3$(ii)$]{HancheStor}), we can find an infinite system $(p_j)_{j\in \Lambda}$ of mutually orthogonal abelian projections such that $\sum_j p_j =\mathbf{1}$, the central support projection of each $p_j$ coincides with the unit of $W$ and card$(\Lambda)=n$. We can also conclude that the $p_j$'s are mutually exchangeable by a symmetry (compare \cite[Lemma 5.3.2]{HancheStor}). Clearly, we can restrict to a countable subfamily.
Then, there is a unital JW-subalgebra of $W$ containing $(p_j)$ that is isomorphic to $B(H)_{sa}^\alpha$, where $\alpha$ is a real flip and $H$ has infinite countable dimension (see \cite[Theorem 7.6.3 $(iii)\Leftrightarrow (iv)$]{HancheStor}). The desired conclusion follows, in this case, from Lemma \ref{l real flip}.\smallskip

By \cite[Theorem 16]{To} (alternatively, \cite[Theorem 5.3.5]{HancheStor}) any properly infinite Type I JW-algebra $W$ can be decomposed into a direct sum of infinite homogeneous ones. We can obtain the desired finite type II$_1$ continuous JW-subalgebra $B$ by taking the sum of all type II$_1$ JW-subfactors given by Lemma \ref{l real flip} in the corresponding homogeneous summand. Actually it is enough to consider a non-zero type II$_1$ JW-subfactor in any of the corresponding homogeneous summands.\smallskip

We assume now that $W$ contains no type I part. Let $p$ be a non-zero projection in $W$. If $p$ is modular then $B=\{p,W,p\}$ is a JW-algebra of type II$_1$, which proves the desired statement. If $p$ is not modular, then $\{p,W,p\}$  contains a copy of $B(H)_{sa}^\alpha$, where $H$ is separable and infinite dimensional, and $\alpha$ is a real flip (see Theorem 7.6.3 $(i)\Leftrightarrow (iv)$ in \cite{HancheStor}). Lemma \ref{l real flip} implies the existence of a type II$_1$ JW-subfactor of $B(H)_{sa}^\alpha$. This finishes the proof of the first statement in Proposition \ref{conjecture subalgebra of type II1}.\smallskip

We consider now the second statement in the proposition. Let $M$ be a continuous von Neumann algebra and suppose $\beta: M\to M$ is a $\mathbb{C}$-linear $^*$-involution. We may assume, without loss of generality, that the type II$_1$ part of $M$ is zero. We consider the JW-algebra $H(M,\beta)_{sa} =\{ a\in M : \beta (a) = a=a^* \}$.\smallskip

We claim that $H(M,\beta)_{sa}$ contains a central projection which is not modular. Let $z$ be a central projection in $H(M,\beta)_{sa}$. If $z$ is not modular the claim is obvious, otherwise $zH(M,\beta)_{sa}z$ is modular. Let $R(M,\beta)=\{ x\in M: \beta(x)=x^* \}$. Clearly, $R(M,\beta)$ is a real von Neumann algebra and $H(M,\beta)_{sa} =\{ x\in R(M,\beta): x=x^* \} $ coincides with the hermitian part of $R(M,\beta)$. We also have  $M=R(M,\beta)+i R(M,\beta),$ via $ x=(x+\beta(x^*))/2 +i(x-\beta(x^*))/2i ).$ We observe that $z$ is a projection in $M$ with $\beta (z) = z^* =z$, $zMz$ is $\beta$-invariant, and $zH(M,\beta)_{sa}z = H(zMz,\beta)_{sa} = R(zMz,\beta)_{sa}$. We deduce from Proposition 1.3 in \cite{Ayu1989} that $R(zMz,\beta)$ is finite. Theorem 2.2 in \cite{Ayu1989} implies that $zMz=R(zMz,\beta)+i R(zMz,\beta)$ (and hence $z$) is finite in $M$. Let $c(z)$ denote the central support projection of $z$ in $M$, that is, $c(z)$ is the smallest central projection in $M$ majorizing $z$. Since $\beta$ is a $\mathbb{C}$-linear $^*$-involution, we deduce that $\beta (c(z))=c(\beta(z))=c(z)$, and thus $c(z)$ lies in $H(M,\beta)_{sa}$. Since the type II$_1$ part of $M$ is zero and $M$ is continuous, we deduce that $c(z)$ must be an infinite central projection in $M$ (compare \cite[Definition V.1.17]{Tak}). Thus $c(z)$ must be a non-modular central projection in $H(M,\beta)_{sa}$, which proves the claim. Indeed, if $c(z)H(M,\beta)_{sa}$ were modular, then as shown above, $c(z)M$ would be finite. \smallskip

Finally, let $p$ be a non-modular central projection in $H(M, \beta)_{sa}$. A new application of \cite[Theorem 7.6.3 $(i)\Leftrightarrow (iv)$]{HancheStor} implies that $\{p, H(M, \beta)_{sa}, p\}=U_p \left(H(M, \beta)_{sa}\right)$ contains a copy of $B(H)_{sa}^\alpha$, where $H$ is separable and infinite dimensional, and $\alpha$ is a real flip. By Lemma \ref{l real flip} there exists a von Neumann algebra $N$ of type II$_1$ such that $N$ is an $\alpha$-invariant von Neumann subalgebra of $B(H)$, $H(N,\alpha)_{sa}$ is a JBW-subalgebra of $H(M,\beta)_{sa}$ and the enveloping von Neumann algebra of $H(N,\alpha)_{sa}$ coincides with $N$. Clearly $N$ is a subalgebra of $M$. Since every $x\in H(N,\alpha)_{sa}$ satisfies $\beta (x) =x=x^*$ and $\beta$ is a $\mathbb{C}$-linear $^*$-involution, the enveloping von Neumann algebra of $H(N,\alpha)_{sa}$, namely $N$, must be $\beta$-invariant, which concludes the proof.
\end{proof}

The remaining results in this section are appropriate adaptations of the corresponding lemmas in \cite{Doro} and \cite[\S 6.1]{Ham03}, they are included here for completeness reasons.\smallskip

Let us observe a simple property.

\begin{remark}\label{remark direct sums} Let $M$ be a von Neumann algebra, let $\beta$ be a $\mathbb{C}$-linear $^*$-involution on $M$, and let $\Delta: \mathcal{P}(H(M,\beta))\to \mathbb{C}$ be a completely additive measure. Suppose we can decompose $M$ as finite direct sum of mutually orthogonal $\beta$-invariant von Neumann subalgebras $M_1,\ldots,M_k$, that is $M=M_1 \oplus^{\infty} \ldots \oplus^{\infty} M_k$ with $\beta(M_j)=M_j$, for every $j$. Then $\Delta$ is bounded if and only if $\Delta|_{\mathcal{P}(H(M_j,\beta))}: \mathcal{P}(H(M_j,\beta))\to \mathbb{C}$ is bounded for every $j=1,\dots,k.$
\end{remark}

Let us briefly recall some basic notions on $\sigma$-finite projections in JBW$^*$-algebras. As in the setting of von Neumann algebras, a JBW$^*$-algebra $M$ is said to be \emph{$\sigma$-finite} if every family of mutually orthogonal non-zero projections in $M$ is at
most countable. A projection $p$ in $M$ is called \emph{$\sigma$-finite} if the JBW$^*$-algebra $U_p (M)$ is $\sigma$-finite, where $U_p$ is the operator on $M$ given by $U_p (x) = \{p,x^{*},p\}= 2 (p \circ x)\circ p -p \circ x$. A projection $p$ in $M$ is $\sigma$-finite if and only if it is the support projection of a normal state in $M_*$ (cf. \cite[Theorem 3.2]{EdRu98}). The supremum of countably many $\sigma$-finite projections is again $\sigma$-finite, and every projection in a JBW$^*$-algebra can be written as a sum of mutually orthogonal $\sigma$-finite projections (see \cite[Theorem 3.4]{EdRu98}). These facts can be derived from \cite{EdRu98}  and are explicitly developed in \cite{BHKArXiv}.\smallskip

The following two results will be applied in several arguments (compare \cite[Lemma 4]{Doro}).

\begin{proposition}\label{p restriction to sigma finite projections} Let $M$ be a continuous von Neumann algebra and let $\beta: M\to M$ be a $\mathbb{C}$-linear $^*$-involution. Suppose that $\Delta: \mathcal{P}(H(M,\beta))\to \mathbb{C}$ is an unbounded completely additive  measure. Then there exists a $\sigma$-finite projection $p\in H(M,\beta)$ such that $\Delta|_{\mathcal{P}(H(pMp, \beta))}$ is unbounded.
\end{proposition}

\begin{proof}
Since $\Delta$ is unbounded, there exists a sequence $(q_n)$ in $\mathcal{P} (H(M,\beta))$ satisfying that $\displaystyle\lim_{n\to \infty} |\Delta (q_n)| =\infty$. Each $q_n$ can be written as the sum of a family of mutually orthogonal $\sigma$-finite projections in $H(M,\beta)$ (compare \cite[Theorem 3.4 $(ii)$]{EdRu98}). Therefore, there exists a family $(p^{(n)}_{\lambda})_{\lambda\in \Lambda_n}$ of mutually orthogonal $\sigma$-finite projections in $H(M,\beta)$ such that $\displaystyle q_n=\sum_{\lambda\in \Lambda_n} p^{n}_{\lambda}$. By the complete additivity of $\Delta$, there exists a finite subset $F_n\subseteq \Lambda_n$ such that
 $$\left| \Delta (q_n) - \Delta \left( \sum_{\lambda\in F_n} p^{(n)}_{\lambda}\right) \right|<\frac{1}{n}.
 $$
 Clearly, $\displaystyle p_n=\sum_{\lambda\in F_n} p^{n}_{\lambda}$ is a $\sigma$-finite projection in $\mathcal{P} (H(M,\beta))$ and $\displaystyle\lim_{n\to \infty} |\Delta (p_n)| =\infty$. Let $p=\bigvee_{n} p_n \in \mathcal{P} (H(M,\beta))$. Since the supremum of countably many $\sigma$-finite projections is again $\sigma$-finite (compare \cite[Theorem 3.4 $(i)$]{EdRu98}), the projection $p$ is $\sigma$-finite, and obviously, $\Delta|_{\mathcal{P}(H(pMp, \beta))}$ is unbounded, which finishes the proof.
\end{proof}

\begin{proposition}\label{proposition l4 Dorofeev Jordan or L 6.1.9 in Hamhalter} Let $M$ be a von Neumann algebra of type II$_1$, II$_\infty$ or III, and let $\beta$ a $\mathbb{C}$-linear $^*$-involution on $M$. Suppose $\Delta: \mathcal{P}(H(M,\beta))\to \mathbb{C}$ is a completely additive unbounded measure. Then there exists a projection $p_0\in \mathcal{P}(H(M,\beta))$ such that $p_0 M p_0$ is $\sigma$-finite and of type II$_1$, II$_\infty$ or III, and the measure $\Delta_0=\Delta|_{\mathcal{P}(H(p_0 M p_0,\beta))}: \mathcal{P}(H(p_0 M p_0,\beta))\to \mathbb{C}$ satisfies the following property \begin{equation}\label{property (1)} \hbox{for each } p\in \mathcal{P}(H(p_0 M p_0,\beta)) \hbox{ with } |\Delta_0 (p)| > 1 \hbox{ the measure }
\end{equation} $$\Delta_0|_{\mathcal{P}(H((p_0-p) M (p_0-p),\beta))}: \mathcal{P}(H((p_0-p) M (p_0-p),\beta))\to \mathbb{C} \hbox{ is bounded.}$$
\end{proposition}

\begin{proof}
If the pair $(H(M,\beta),\Delta)$ satisfies the desired property then the proof is concluded by taking $p_0=\mathbf{1}$. Otherwise, there exists a projection $p_1\in \mathcal{P}(H(M,\beta))$ with $|\Delta (p_1)| > 1$ satisfying that the measure $\Delta|_{\mathcal{P}(H((\mathbf{1}-p_1) M (\mathbf{1}-p_1),\beta))}: \mathcal{P}(H((\mathbf{1}-p_1) M (\mathbf{1}-p_1),\beta))\to \mathbb{C}$ is unbounded. Since $(\mathbf{1}-p_1) M (\mathbf{1}-p_1)$ doesn't contain type I part, we can decompose $(\mathbf{1}-p_1) M (\mathbf{1}-p_1)$ as a direct sum of von Neumann subalgebras of type II$_1$, II$_\infty$ or III. We also observe that each summand in the above decomposition must be $\beta$-invariant. Therefore, by Remark \ref{remark direct sums}, there exists a subprojection $q_1\leq \mathbf{1}-p_1$ such that $q_1 M q_1$ is of type II$_1$, II$_\infty$ or III and $\Delta|_{\mathcal{P}(H(q_1 M q_1,\beta))}$ is unbounded.\smallskip

If the pair $(H(q_1 M q_1,\beta), \Delta|_{\mathcal{P}(H(q_1 M q_1,\beta))})$ satisfies property \eqref{property (1)} we obtain the desired statement. Otherwise, applying the above argument, there exists $p_2 \leq q_1$ in $H(q_1 M q_1,\beta)$ such that $|\Delta (p_2)| > 1$ and $\Delta|_{\mathcal{P}(H((q_1-p_2) M (q_1-p_2),\beta))}: \mathcal{P}(H((q_1-p_2) M (q_1-p_2),\beta))\to \mathbb{C}$ is unbounded. Thus, there exists $q_2\leq q_1-p_2$ such that  $q_2 M q_2$ is of type II$_1$, II$_\infty$ or III and $\Delta|_{\mathcal{P}(H(q_2 M q_2,\beta))}$ is unbounded.\smallskip

By repeating the above arguments, we find a pair  $(q_n M q_n, \Delta|_{\mathcal{P}(H(q_n M q_n,\beta))})$ (with $\beta(q_n)=q_n$, for every $n\in\mathbb{N}$) satisfying the desired statement, or there exists an infinite sequence $(p_n)$ of mutually orthogonal $\beta$-symmetric projections in $M$ satisfying $|\Delta (p_n)| > 1,$ for every natural $n$, which contradicts the complete additivity of $\Delta$.
\end{proof}

Henceforth, up to and including Lemma~\ref{lemma 9 Jordan Dorofeev 6.1.15 Ham}, $M$ will denote a $\sigma$-finite von Neumann algebra of type II$_1$, II$_\infty$ or III, $\beta$ a $\mathbb{C}$-linear $^*$-involution on $M$, and $N$ a type II$_1$ von Neumann subalgebra of $M$ satisfying $\beta(N)=N$ (compare Proposition \ref{conjecture subalgebra of type II1}). We observe that $H(N,\beta)$ is a JBW$^*$-subalgebra of $H(M,\beta)$. From now on, $\tau$ will stand for a faithful normal norm-one finite trace on $N$, whose restriction to $H(N,\beta)$ will be also denoted by $\tau$.

First, we recall some facts about the strong$^*$ topology.  For each normal positive functional $\varphi$ in the predual of a von Neumann algebra $M$, the mapping $$x\mapsto \|x\|_{\varphi} = \left(\varphi (\frac{x x^* + x^* x}{2}) \right)^{\frac12}\quad (x\in M)$$ defines a prehilbertian seminorm on $M$. The \emph{strong$^*$ topology} of $M$, denoted by $S^*(M,M_*)$,  is the locally convex topology on $M$ defined by all the seminorms $\|.\|_{\varphi}$, where $\varphi$ runs in the set of all positive functionals in $M_*$

\begin{lemma}\label{lemma 5 Dorofeev or 6.1.11 Ham} Let $\Delta: \mathcal{P}(H(M,\beta))\to \mathbb{C}$ be a completely additive unbounded measure. Suppose that the pair $(H(M,\beta),\Delta)$ satisfies property \eqref{property (1)} in Lemma \ref{proposition l4 Dorofeev Jordan or L 6.1.9 in Hamhalter}. Let $N$ a type II$_1$ von Neumann subalgebra of $M$ satisfying $\beta(N)=N$, and let $\tau$ denote the unital normal faithful finite trace on $N$. Then there exist a positive constant $K$ and $0<\delta< 1$ satisfying the following property: \begin{equation}\label{properpty (2)} \hbox{For each } q\in H(N,\beta) \hbox{ with } \tau(q)\leq \delta, \hbox{ we have }  \sup\{ |\Delta (p)| : p\in H(M,\beta), p\leq q\} \leq K.
\end{equation}
\end{lemma}

\begin{proof} (compare \cite[Lemma 5]{Doro})  Arguing by reduction to the absurd, we suppose that the desired property does not hold. Then there exists a sequence $(q_n)$ in  $H(N,\beta)$ such that $|\tau (q_n)| \leq \frac{1}{2^n}$ and $$\sup\{ |\Delta (p)| : p\in H(M,\beta), p\leq q_n\}>n,$$ for every natural $n$.  Set $G_n:=\vee_{k=n}^{\infty} q_k$. Since every $q_n$ is $\beta$-symmetric, we deduce that $G_n$ also is $\beta$-symmetric for all $n\in\mathbb{N}$ (i.e., $(G_n)\subset H(N,\beta)$). Considering strong$^*$-limits of growing sequences, we deduce that $\tau(G_n)\leq \sum_{k=n}^{\infty} \tau(q_k) \leq \sum_{k=n}^{\infty} \frac{1}{2^k}$, which implies, by the faithfulness of $\tau$ on $N$, that $(G_n)\searrow 0$ in the strong$^*$-topology of $N$ (and also in the strong$^*$-topology of $M$). Since $G_n\geq q_n$ and $G_n\downarrow 0$, we have for every $m\ge n$,
\[
\sup\{|\Delta(p)|;p\le G_n\}\ge \sup\{|\Delta(p)|;p\le G_m\}\ge \sup\{|\Delta(p)|;p\le q_n\}>m,
\]
and $\Delta|_{\mathcal{P}(H(G_n M G_n,\beta))}$ is unbounded. Since the pair $(H(M,\beta),\Delta)$ satisfies property \eqref{property (1)}, we deduce that $\sup\{ |\Delta (p)| : p\in H(M,\beta), p\leq \mathbf{1}-G_n\}\leq 1$, for every natural $n$. Since $\mathbf{1}-G_n\nearrow \mathbf{1}$ in the strong$^*$-topology, Lemma \ref{l 2 Dorofeev in Jordan} implies that $\Delta$ is bounded, which is impossible.
\end{proof}

The automatic boundedness of completely additive measure on $\mathcal{P}(H(M,\beta))$ actually relies on the appropriate Jordan version of the Mackey-Gleason theorem stated in Theorem \ref{t Jordan Mackey-Gleason}.

\begin{lemma}\label{lemma 6 Dorofeev or 6.1.13 Ham} Let $\Delta: \mathcal{P}(H(M,\beta))\to \mathbb{C}$ be a completely additive unbounded measure. Suppose that the pair $(H(M,\beta),\Delta)$ satisfies the property \eqref{property (1)} in Lemma \ref{proposition l4 Dorofeev Jordan or L 6.1.9 in Hamhalter}. Let $N$ a type II$_1$ von Neumann subalgebra of $M$ satisfying $\beta(N)=N$ and let $\tau$ be the unital faithful normal trace on $N$. Then there exists a positive constant $C_0$ satisfying that if $q\in \mathcal{P}(H(N,\beta))$ and $\Delta|_{\mathcal{P}(H(qMq,\beta))}$ is bounded then $$\sup \{|\Delta(p)| : p\in \mathcal{P}(H(qMq,\beta))\} \leq C_0.$$
\end{lemma}

\begin{proof} The proof given in \cite[Lemma 6]{Doro} or in \cite[Lemma 6.1.13]{Ham03} remains valid here when we replace \cite[Lemma 5]{Doro} with our previous Lemma \ref{lemma 5 Dorofeev or 6.1.11 Ham} and the Bunce-Wright-Mackey-Gleason theorem \cite{BuWri92} with Theorem \ref{t Jordan Mackey-Gleason}. We include some details for completeness reasons.\smallskip

Let $C_0=\frac{32 K}{\delta}$, where $K$ and $\delta$ are given by Lemma \ref{lemma 5 Dorofeev or 6.1.11 Ham}. Take a projection $q$ in $H(N,\beta)$ with $\Delta|_{\mathcal{P}(H(qMq,\beta))}$ bounded. Theorem \ref{t Jordan Mackey-Gleason} implies the existence of a (normal) continuous linear functional $\varphi: H(q M q,\beta) \to \mathbb{C}$ such that $\varphi(p) =\Delta (p)$, for every $p\in \mathcal{P}(H(qMq,\beta))$. By Lemma \ref{l 3 Jordan} there exists a family of pairwise orthogonal projections $(q_i)_{i=1,\ldots,n}$ in $H(N,\beta)$ such that $q= \sum_{i=1}^{n} q_i,$ $\tau (q_i) \leq \frac{\delta}{2},$ for every $i=1,\ldots, n,$ and $n\leq 4/\delta.$\smallskip

Let us pick an arbitrary projection $p\in H(qMq,\beta)$. We need to show that $|\Delta(p)|\le C_0$. If we write $p =\sum_{i,j=1}^{n} \{q_i, p, q_j\}=\sum_{i,j=1}^{n} \frac12(q_i p q_j+q_j p q_i)$, we observe that, for $i\neq j$, $q_i +q_j$ is a projection in $H(N,\beta)$ and $\tau (q_i+q_j) \leq \delta$. Lemma \ref{lemma 5 Dorofeev or 6.1.11 Ham} implies that

 $$\sup \{ |\Delta (r) |  : r\in \mathcal{P} (H(M,\beta)), r\leq q_i+q_j \} \leq K, $$ and hence that
 $$ \sup \{|\varphi(r)| : r\in \mathcal{P} (H(M,\beta)), r\leq q_i+q_j \} \leq K.$$


 Considering spectral resolutions, we deduce that $\sup \{ |\varphi(a)| : a \in H((q_i+q_j)M(q_i+q_j),\beta), \|a\|\leq 1 \} \leq 2 K$. Similarly, $\sup \{ |\varphi(a)| : a \in H(q_j M q_j,\beta), \|a\|\leq 1 \} \leq 2 K$, for every $j=1,\dots,n.$ Therefore, having in mind that $q_i p q_j+q_j p q_i$ lies in $H((q_i+q_j)M(q_i+q_j),\beta)$, we deduce that $$|\Delta(p)|=|\varphi(p)|  \leq \sum_{i,j=1}^{n} \frac12 |\varphi(q_i p q_j+q_j p q_i)| \leq n^2 2 K \leq \frac{16}{\delta^2} 2 K =C_0.$$
\end{proof}

In a similar fashion, replacing \cite[Lemmas 2 and 6]{Doro} with Lemmas \ref{l 2 Dorofeev in Jordan} and \ref{lemma 6 Dorofeev or 6.1.13 Ham}, respectively, the proof of \cite[Lemma 7]{Doro} holds to prove the following result.

\begin{lemma}\label{lemma 7 Dorofeev or 6.1.14 Ham} Let $\Delta: \mathcal{P}(H(M,\beta))\to \mathbb{C}$ be a completely additive unbounded measure, where $H(M,\beta)$ is a $\sigma$-finite JBW$^*$-algebra. Suppose that the pair $(H(M,\beta),\Delta)$ satisfies the property \eqref{property (1)} in Lemma \ref{proposition l4 Dorofeev Jordan or L 6.1.9 in Hamhalter}. Let $N$ a type II$_1$ von Neumann subalgebra of $M$ satisfying $\beta(N)=N$. Then there exists a projection $q_0$ in $H(N,\beta)$ satisfying the following properties:\begin{enumerate}[$(a)$]\item $\Delta|_{\mathcal{P}(H(q_0 M q_0,\beta))}$ is bounded;
\item if $q\in \mathcal{P} (H(N,\beta))$, $q\gvertneqq q_0$ then $\Delta|_{\mathcal{P}(H(q M q,\beta))}$ is unbounded. 
\end{enumerate}
\end{lemma}

\begin{proof} Let $\mathcal{B}$ denote the set of all families $(q_i)_{i\in I}$ of mutually orthogonal projections in $H(N,\beta)$ such that for each finite subset $F\subset I$, the projection $q_F :=\sum_{i\in F} q_i$ satisfies that $\Delta|_{\mathcal{P} (H(q_F M q_F,\beta))}$ is bounded. The set $\mathcal{B}$ is an inductive set when it is equipped with the order given by inclusion (by Proposition~\ref{proposition l4 Dorofeev Jordan or L 6.1.9 in Hamhalter}, $\mathcal{B}\ne\emptyset$). By Zorn's lemma there exists a maximal element $(q_i^{0})_I\in \mathcal{B}$. The set I is at most countable because $H(M,\beta)$ is $\sigma$-finite. We claim that the projection $q_0=\sum_{i\in I} q_i^0\in H(N,\beta)$ satisfies the desired property. Indeed, defining $q_n:=\sum_{i=1}^n q_i^0,$ we have $q_n \nearrow q_0$. Since $(q_i^{0})_I\in \mathcal{B}$, the measure $\Delta|_{\mathcal{P}(H(q_n M q_n,\beta))}$ is bounded for every $n$. Lemma \ref{lemma 6 Dorofeev or 6.1.13 Ham} implies the existence of a constant $C_0>0$ such that $\sup\{|\Delta(p)| : p\in H(M,\beta), p\leq q_n \}\leq C_0,$ for every natural $n$. Lemma \ref{l 2 Dorofeev in Jordan} proves that $\Delta|_{\mathcal{P}(H(q_0 M q_0,\beta))}$ is bounded.\smallskip

Finally, the second property follows from the maximality of the element $(q_i^{0})_I\in \mathcal{B}$.
\end{proof}

We shall see now that the arguments in the proof of \cite[Lemma 6.1.15]{Ham03} are also valid in the Jordan setting. Actually, the proof follows the arguments we gave in Lemma \ref{lemma 5 Dorofeev or 6.1.11 Ham}.

\begin{lemma}\label{lemma 6.1.15 Ham} Let $\Delta: \mathcal{P}(H(M,\beta))\to \mathbb{C}$ be a completely additive unbounded measure. Suppose that the pair $(H(M,\beta),\Delta)$ satisfies the property \eqref{property (1)} in Lemma \ref{proposition l4 Dorofeev Jordan or L 6.1.9 in Hamhalter}. Let us assume that $H(M,\beta)$ is $\sigma$-finite and let $\varphi$ be a faithful normal state on $H(M,\beta)$. Then there exists a projection $p_0$ in $H(M,\beta)$ and $\delta>0$ such that $\Delta|_{\mathcal{P}(H(p_0 M p_0,\beta))}$ is unbounded and the following property holds: \begin{equation}\label{eq property(3)} \hbox{If } p\in \mathcal{P} (H(M,\beta)),  p\leq p_0 \hbox{ and } \varphi (p)\leq \delta, \hbox{ then } \Delta|_{\mathcal{P}(H(p M p,\beta))} \hbox{ is bounded.}
\end{equation}
\end{lemma}

\begin{proof} If the desired property holds for $p_0=\mathbf{1}$ and some $\delta$, then the Lemma is proved. Otherwise, there exists a projection $p_1$ in $\mathcal{P} (H(M,\beta))$ such that $\varphi (p_1)\leq \frac12$ and $\Delta|_{\mathcal{P}(H(p_1 M p_1,\beta))}$ is unbounded. If $p_1$ satisfies the desired property the statement is proved. If that is not the case, there exists a projection $p_2$ in $\mathcal{P} (H(M,\beta))$ such that $p_2\leq p_1$, $\varphi (p_2)\leq \frac13$ and $\Delta|_{\mathcal{P}(H(p_2 M p_2,\beta))}$ is unbounded. Repeating the above argument, we obtain the desired conclusion for a suitable projection, or there exists a decreasing sequence of projections $(p_n)$ in $H(M,\beta)$ satisfying $\varphi (p_n)\leq \frac1n$ and $\Delta|_{\mathcal{P}(H(p_n M p_n,\beta))}$ is unbounded. The faithfulness of $\varphi$ implies that $p_n\searrow 0$ in the strong$^*$-topology.\smallskip

Since the pair $(H(M,\beta),\Delta)$ satisfies property \eqref{property (1)} in Lemma \ref{proposition l4 Dorofeev Jordan or L 6.1.9 in Hamhalter}, we conclude that $$\sup\{ |\Delta (p)| : p\in \mathcal{P} (H(M,\beta)), p\leq 1-p_n\} \leq 1,$$ for all $n$. Recalling that $\mathbf{1}-p_n \nearrow\mathbf{1}$ in the strong$^*$-topology, Lemma \ref{l 2 Dorofeev in Jordan} implies that $\Delta$ is bounded, which contradicts the hypothesis of the lemma.
\end{proof}

\begin{lemma}\label{lemma 9 Jordan Dorofeev 6.1.15 Ham} Let $\Delta: \mathcal{P}(H(M,\beta))\to \mathbb{C}$ be a completely additive unbounded measure. Let $p_0$ be a projection in $H(M,\beta)$ satisfying that $\Delta|_{\mathcal{P}(H(p_0 M p_0,\beta))}$ and $\Delta|_{\mathcal{P}(H((1-p_0) M (1-p_0),\beta))}$ are bounded. Let $K_n\rightarrow\infty$.  Then for each natural $n$, there exists a projection $q=q_n\in H(M,\beta)$ such that $|\Delta (q)|>K_n$ and $d(q,p_0)\geq \frac18$.
\end{lemma}

\begin{proof} (compare \cite[Lemma 9]{Doro})   Let us take $C>0$ satisfying $\sup \{ |\Delta(p)| : p\in H(M,\beta), p\leq p_0 \hbox{ or } p\leq \mathbf{1}-p_0\}\leq C$. Given $n$, by the unboundedness of $\Delta$, we can find a projection $p$ in $H(M,\beta)$ such that $|\Delta (p)|> 2 K_n + 6 C$. The projection $r=p+(\mathbf{1}-p)\wedge p_0\in H(M,\beta)$ satisfies $(\mathbf{1}-r) \wedge p=0$ and $|\Delta(r)|\geq |\Delta(p)|- |\Delta((\mathbf{1}-p)\wedge p_0)|> 2 K_n +5 C$. By \cite[Remark 1]{Doro} we have $$r=r(r p_0 r) + r\wedge (\mathbf{1}-p_0)$$ in $M$, as well as in $H(M,\beta)$. The proof of \cite[Lemma 9]{Doro} shows that taking $r_1=1_{(0,\frac12]} (r p_0 r)\in H(M,\beta)$ and $r_2=1_{(\frac12,1]} (r p_0 r)\in H(M,\beta)$ we have $d(r_1,p_0) \geq \frac12$ and $r_1+r_2 = r -r\wedge (\mathbf{1}-p_0)$. It is further seen that for $r_2' =r_2 \vee (\mathbf{1}-p_0) -r_2\in H(M,\beta)$ the inequality $d(r_2',p_0) \geq \frac18$ holds.

It is also clear that $r_1 \perp r_2$, and since $r-r\wedge (\mathbf{1}-p_0) \perp r\wedge (\mathbf{1}-p_0)$. Therefore $$|\Delta (r_1)| +|\Delta(r_2)|  \geq |\Delta (r_1) +\Delta(r_2)| = |\Delta (r_1 + r_2)|=|\Delta (r- r\wedge (1-p_0))|$$ $$=|\Delta (r) - \Delta (r\wedge (\mathbf{1}-p_0))| \geq |\Delta (r)| - |\Delta (r\wedge (\mathbf{1}-p_0))|> 2K_n +4 C.$$ It follows that $|\Delta (r_1)| > K_n+ 2C$ or $|\Delta (r_2)| > K_n+ 2C$. In the first case the projection $q=r_1$ satisfies the desired statement; otherwise, the projection $q=r_2'$ satisfies the conclusion of the lemma. Indeed, $$|\Delta (r_2\vee (\mathbf{1}-p_0)) | \leq |\Delta(p_0)| + |\Delta (r_2\vee (\mathbf{1}-p_0) -(\mathbf{1}-p_0)) | \leq 2 C,$$ because $(r_2\vee (\mathbf{1}-p_0) -(\mathbf{1}-p_0))\perp (\mathbf{1}-p_0)$. Thus, we get $$|\Delta(q)|=|\Delta(r_2')| \geq |\Delta(r_2)| - |\Delta(r_2\vee (\mathbf{1}-p_0))| > K_n.$$
\end{proof}

We complete now the proof of our Jordan version of Dorofeev's theorem. The arguments are based on appropriate Jordan adaptations of the proofs in \cite[Theorem 1]{Doro} and \cite[Theorem 6.1.16]{Ham03}.

\begin{proof}[Proof of Theorem \ref{t Jordan Dorofeev-Sherstnev}] Arguing by contradiction, we shall assume that $\Delta: \mathcal{P}(H(M,\beta))\to \mathbb{C}$ is an unbounded completely additive measure. By Proposition \ref{p restriction to sigma finite projections} there exists a $\sigma$-finite projection $p\in H(M,\beta)$ such that $\Delta|_{\mathcal{P}(H(pMp, \beta))}$ is unbounded.\smallskip

We can therefore assume that $H(M,\beta)$ is $\sigma$-finite. Let $\varphi$ be a faithful normal state on $H(M,\beta)$. Furthermore, by Remark \ref{remark direct sums}, we can also assume that $M$ is of type II$_1$, II$_\infty$ or III.\smallskip

Having in mind Proposition \ref{proposition l4 Dorofeev Jordan or L 6.1.9 in Hamhalter}, we can assume that the pair $(H(M,\beta),\Delta)$ satisfies property \eqref{property (1)} for $p_0=\mathbf{1}$ in that proposition  (otherwise we replace $M$ with $p_0 M p_0$). Applying Lemma \ref{lemma 6.1.15 Ham}, we may assume that $\Delta$ satisfies property \eqref{eq property(3)} for $p_0=\mathbf{1}$, the faithful normal state $\varphi$ fixed in the above paragraph, and a suitable $\delta>0$. By Proposition \ref{conjecture subalgebra of type II1} there exists a type II$_1$ subalgebra $N$ of $M$ such that $\beta (N) =N$.\smallskip

Let $q_0$ be the projection in $H(N,\beta)$ given by Lemma \ref{lemma 7 Dorofeev or 6.1.14 Ham}, that is, $q_0$ satisfies the following properties: \begin{enumerate}[$(a)$]\item $\Delta|_{\mathcal{P}(H(q_0 M q_0,\beta))}$ is bounded;
\item if $q\in \mathcal{P} (H(N,\beta))$, $q\gvertneqq q_0$ then $\Delta|_{\mathcal{P}(H(q M q,\beta))}$ is unbounded. 
\end{enumerate}

The unboundedness of $\Delta$ implies that $q_0\neq \mathbf{1}$. By the Halving lemma (see \cite[Theorem 5.2.14]{HancheStor}) there exists an infinite sequence $(q_n)$ of mutually orthogonal nonzero projections in $H(N,\beta)$ such that $q_n\leq \mathbf{1}-q_0$, for every $n\in \mathbb{N}.$ Property $(b)$ above implies that $\Delta|_{\mathcal{P}(H((q_0+q_n) M (q_0+q_n),\beta))}$ is unbounded for all natural $n$.\smallskip

We claim that $\Delta|_{\mathcal{P}(H((\mathbf{1}-q_0) M (\mathbf{1}-q_0),\beta))}$ is bounded. Indeed, let $(r_n)$ be a sequence of projections in $H(N,\beta)$ such that $(r_n)\searrow 0$ and $r_n\leq\mathbf{1}-q_0$. The above property $(b)$ of $q_0$ also implies that $\Delta|_{\mathcal{P}(H((q_0+r_n) M (q_0+r_n),\beta))}$ is unbounded for all natural $n$. Since the pair $(H(M,\beta), \Delta)$ satisfies property \eqref{property (1)} for $p_0=\mathbf{1}$ in Proposition \ref{proposition l4 Dorofeev Jordan or L 6.1.9 in Hamhalter}, it follows that $$\sup \{ |\Delta(p)| : p\leq \mathbf{1}-q_0-r_n \}\leq 1, \ (n\in\mathbb{N}).$$ The boundedness of the previous set together with the condition $\mathbf{1}-q_0-r_n \nearrow \mathbf{1}-q_0$ imply, via Lemma \ref{l 2 Dorofeev in Jordan} that $\Delta|_{\mathcal{P}(H((\mathbf{1}-q_0) M (\mathbf{1}-q_0),\beta))}$ is bounded, which proves the claim.\smallskip

We have shown that $\Delta|_{\mathcal{P}(H(q_0 M q_0,\beta))}$ and $\Delta|_{\mathcal{P}(H((\mathbf{1}-q_0) M (\mathbf{1}-q_0),\beta))}$ are bounded measures. Applying Lemma \ref{lemma 9 Jordan Dorofeev 6.1.15 Ham} to $(q_0+q_n) M (q_0+q_n)$ and the projection $p_0=q_0$, we find a projection $p_n$ in $(q_0+q_n) M (q_0+q_n)$ satisfying $|\Delta(p_n)| > \frac{n 2^{12 n}}{\delta}$ and $d(p_n, q_0)\geq \frac18$ (let us observe that since $q_n \leq \mathbf{1}-q_0$,  $\Delta|_{\mathcal{P}(H(q_n M q_n,\beta))}$ is bounded). We define in this way a sequence $(p_n)$ in $\mathcal{P} (H(M,\beta))$.\smallskip

We shall prove next that, for each natural $n$, $d(p_n, \vee_{i\neq n} p_i)\geq \frac18$. To this end, let us pick norm-one elements $\xi\in p_n(H)$ and $\eta\in \vee_{i\neq n} p_i (H)$ (we regard $M$ as a von Neumann subalgebra of some $B(H)$). Having in mind that $p_n \leq q_0 + q_n$ with $q_n\perp q_0$ ($n\in \mathbb{N}$), we deduce that $ \vee_{i\neq n} p_i (H) \subset  q_0 (H) + \left(\sum_{i\neq n} q_i\right) (H)$, and thus, we can write $$\eta = \alpha u_1 + \beta u_2,$$ where $\alpha,\beta\geq 0$, $\alpha^2 +\beta^2 =1$, $u_1\in q_0(H)$ and $u_2\in \left(\sum_{i\neq n} q_i\right) (H)$. The images of $q_0$ and $\left(\sum_{i\neq n} q_i\right)$ are orthogonal in the Hilbert sense, and hence $$\|\xi-\eta \|^2 = \| \xi - \alpha u_1 - \beta u_2\|^2 = \| \xi - \alpha u_1 \|^2 + \| \beta u_2\|^2 $$ $$\geq \left(\|\xi-u_1 \|- \|(1-\alpha) u_1\| \right)^2 + \beta^2 = \left(\|\xi-u_1 \|- 1+\alpha \right)^2 +1- \alpha^2.$$ The last expression in the above inequality defines a function $f(\alpha)$, $\alpha\in [0,1]$, whose extreme values are attained at $\alpha=0$ or $\alpha=1$. Taking $\alpha=0$, we have $\|\xi-\eta \|^2 \geq \left(\|\xi-u_1 \|- 1 \right)^2 +1 \geq 1$. In the case $\alpha=1$, we have $\|\xi-\eta \|^2 = \|\xi-u_1 \|^2\geq \frac{1}{8^2}$, because $u_1\in q_0(H)$ and $d(q_0, p_n)\geq \frac18$.\smallskip

We apply now Lemma \ref{l 3 Jordan}. For each natural $n$, we can find a finite set $\{p_i^n : i=1,\ldots, k_n\}$ of mutually orthogonal projections in $H(M,\beta)$ satisfying $p_n =\sum_{i=1}^{k_n} p_i^n$, $\varphi (p_i^n) \leq \frac{\delta}{2^{11n}}$, and  $k_n \leq 2 \frac{2^{11 n}}{\delta}$. The projections in $\{p_i^n : i=1,\ldots, k_n\}$ are mutually orthogonal, so $$\frac{n 2^{12n}}{\delta} < |\Delta (p_n)| =\left| \sum_{i=1}^{k_n} \Delta(p_i^n) \right|\leq  \sum_{i=1}^{k_n} \left|\Delta(p_i^n) \right|,$$ and therefore there exists $i_n \in \{1,\ldots, k_n\}$ such that $|\Delta(p_{i_n}^n)| >n$. So, replacing $p_n$ with $p_{i_n}^n$, it may be assumed that $\varphi (p_n) \leq \frac{\delta}{2^{11n}}$ and $|\Delta(p_{n})| >n$.\smallskip

Now, we take $\varepsilon =\frac{1}{2^{10}}$. Lemma 1$(b)$ in \cite{Doro} asserts that $$p_1\vee \ldots \vee p_n \leq \frac{1}{\varepsilon} (p_1+p_2\vee \ldots \vee p_n) \leq \frac{1}{\varepsilon} p_1 + \frac{1}{\varepsilon^2} (p_2 + p_3\vee \ldots \vee p_n) \leq \ldots \leq \sum_{k=1}^{n} \frac{1}{\varepsilon^k} p_k.$$ Therefore,
$$\varphi(p_1\vee \ldots \vee p_n) \leq \sum_{k=1}^{n}  \frac{1}{\varepsilon^k} \varphi(p_k)\leq \sum_{k=1}^{n}  \frac{1}{\varepsilon^k} \frac{\delta}{2^{11k}}= \sum_{k=1}^{n}  2^{10 k} \frac{\delta}{2^{11k}}< \delta.
$$
This shows that for $\displaystyle r=\vee_{n=1}^{\infty} p_n$, $\varphi(r) \leq  \delta$ and $\Delta|_{\mathcal{P}(H(r M r,\beta))}$ is unbounded, which contradicts that $\Delta$ satisfies property \eqref{eq property(3)} for $p_0=1$ and $\delta>0$.
\end{proof}

\section{2-local triple derivations on continuous JBW$^*$-triples}\label{section5}

Recall that a $JBW^\ast$-triple $\mathcal{A}$ is said to be
\textit{continuous} if it has no type I direct summand, and that in this
case, up to isometry, $\mathcal{A}$ is a $JW^\ast$-triple with
unique decomposition, $\mathcal{A} = H(W,\alpha)\oplus pV,$ where $W$
and $V$ are continuous von Neumann algebras, $p$ is a projection
in $V$, $\alpha$ is  a $^*$-antiautomorphism of $W$ of order 2,  and
$H(W,\alpha)=\{x\in W: \alpha(x)=x\}$ (see \cite[(1.20)]{Horn88}).

We have shown in section~\ref{section3} that every 2-local triple derivation on $pV$ is a triple derivation.  In this section we show that every 2-local triple derivation on $H(W,\alpha)$ is a triple derivation, and hence that every
 2-local triple derivation on a continuous JBW$^*$-triple is a triple derivation.

\subsection{Triple derivations on $H(M,\beta)$}


Assume that $M$ is a continuous
von Neumann algebra and $\beta: M\to M$ is a $\mathbb{C}$-linear
$^*$-involution (i.e. a $^*$-antiautomorphism of order 2). In this subsection we shall show that every 2-local triple derivation on the subspace $H(M,\beta)$ of all $\beta$-fixed
points in $M$ is a triple derivation. \smallskip

We begin by taking advantage of the Jordan structure of $H(M,\beta)$ (see the beginning of section~\ref{susect: Dorofeev}) to
provide a precise description of triple derivations on
it.\smallskip

Let $\delta: H(M,\beta) \to H(M,\beta)$ be a triple derivation. By
\cite[Lemma 1 and its proof]{HoMarPeRu},
\begin{equation}\label{equ delta 1 is skew} \delta(\mathbf{1})^* =
-\delta(\mathbf{1}), \hbox{ and } M_{\delta(\mathbf{1})} = \delta\left(\frac12
\delta(\mathbf{1}),\mathbf{1}\right) \hbox{ is a triple derivation.}
\end{equation} This implies that $D= \delta -M_{\delta(\mathbf{1})} = \delta - \delta\left(\frac12 \delta(\mathbf{1}),\mathbf{1}\right)$ is a triple derivation satisfying $D(\mathbf{1}) = 0$. Lemma 2 in \cite{HoMarPeRu} implies that $D$ is a Jordan $^*$-derivation on $H(M,\beta)$. Thus, $D|_{H(M,\beta)_{sa}}: H(M,\beta)_{sa} \to H(M,\beta)_{sa}$ is a Jordan derivation on the continuous JBW-algebra $H(M,\beta)_{sa}$. Theorem 3.5 in \cite{Upmeier80} assures that $D|_{H(M,\beta)_{sa}}$ is an inner derivation, that is, there exist $a_1,\ldots,a_m$ $b_1,\ldots, b_m$ in $H(M,\beta)_{sa}$ satisfying \begin{equation}
\label{equ inner derivation} D (x) = \sum_{j=1}^{m}
\left[M_{a_j},M_{b_j}\right] (x) = \sum_{j=1}^{m}
{a_j}\circ({b_j}\circ x) - {b_j}\circ({a_j}\circ x)
\end{equation} $$= \frac14  \sum_{j=1}^{m} (a_j b_j - b_j a_j) x - x(a_j b_j - b_j a_j) =  \sum_{j=1}^{m} \left[\frac{(a_j b_j - b_j a_j)}{4}, x \right]  =  \left[\sum_{j=1}^{m} \frac{(a_j b_j - b_j a_j)}{4}, x \right],$$ for every $x\in H(M,\beta)_{sa}$. If we denote $\displaystyle a =\sum_{j=1}^{m} \frac{(a_j b_j - b_j a_j)}{4}\in M$, then $\beta(a) = -a$ and $a^* = -a$ (just observe that  $\beta(a_j) = a_j$, $a_j^* = a_j$, $\beta(b_j) = b_j$, and $b_j^* = b_j$, for every $j$), and, by \eqref{equ inner derivation}, $$\delta (x) = [a,x] + \delta(1) \circ x,$$ for every $x\in {H(M,\beta)_{sa}}$. The following proposition summarizes the above facts. \smallskip

\begin{proposition}\label{structure of triple derivations on continuous type 2} Let $M$ be a continuous von Neumann algebra and let $\beta: M\to M$ be a $\mathbb{C}$-linear $^*$-involution. Then for every triple derivation $\delta$ on the JBW$^*$-algebra $H(M,\beta)$, of all $\beta$-fixed points in $M$, there exist $a,b\in M$ with $a^* =-a$, $b^*= -b$, $\beta (a) = -a$ and $\beta (b) = b= \delta(\mathbf{1}),$ satisfying $$\delta (x)= [a,x] + b \circ x,$$ for every $x\in H(M,\beta)$. Consequently, every triple derivation on $H(M,\beta)$ admits an extension to a triple derivation on $M$.$\hfill\Box$\end{proposition}

\subsection{2-local triple derivations on $H(M,\beta)$}

Let $J$ be a JBW$^*$-subalgebra of a von Neumann algebra $M$.
Suppose that $J$ contains the unit of $M$. Given a self-adjoint
element $z\in J$, the JBW$^*$-subalgebra, $\mathcal{W}^\ast(z)$,
of $J$ generated by $z$ and the unit element is an associative
JBW$^*$-algebra isometrically isomorphic to a commutative von
Neumann algebra (cf. \cite[Lemma 4.1.11]{HancheStor}). It is known
that $\mathcal{W}^\ast(z)$ coincides with the abelian von Neumann
subalgebra of $M$ generated by the element $z$ and the unit
element.\smallskip

Let $\Delta: H(M,\beta) \to H(M,\beta)$ be a {\rm(}not necessarily
linear nor continuous{\rm)} 2-local triple derivation. By
\eqref{equ delta 1 is skew} we deduce that $\Delta(\mathbf{1})^* =
-\Delta(\mathbf{1})$ and $M_{\Delta(\mathbf{1})} = \delta\left(\frac12
\Delta(\mathbf{1}),\mathbf{1}\right)$ is a triple derivation. Replacing $\Delta$
with $\Delta-\delta\left(\frac12 \Delta(\mathbf{1}),\mathbf{1}\right)$ we can assume
that our 2-local triple derivation satisfies $\Delta(\mathbf{1})=0$. Having
in mind the description provided by the above Proposition
\ref{structure of triple derivations on continuous type 2}, the
arguments given in \cite[Lemmas 2.2, 2.3, and 2.6]{KOPR2014} can
be literally adapted to prove the following:

\begin{lemma}\label{l 2.2, 2.3 and 2.6 in one} Let $M$ be a continuous von Neumann algebra and let $\beta: M\to M$ be a $\mathbb{C}$-linear $^*$-involution. Suppose that $\Delta: H(M,\beta) \to H(M,\beta)$ is a {\rm(}not necessarily linear nor continuous{\rm)} 2-local triple derivation. Then the following statements hold:\begin{enumerate}[$(a)$]\item If $\Delta(\mathbf{1})=0,$ then $\Delta(x)=\Delta(x)^\ast$ for all $x\in  H(M,\beta)_{sa};$
\item If $\Delta(\mathbf{1})=0,$ then for every $x, y\in H(M,\beta)_{sa}$ there exists a skew-hermitian element $a_{x,y}\in M$ with $\beta (a_{x,y}) = -a_{x,y}$ such that $\Delta (x)=[a_{x,y},x],$ and $\Delta(y)=[a_{x,y},y];$
\item  Let $z\in H(M,\beta)$ be a self-adjoint element and let $\mathcal{W}^\ast(z)=\{z\}''$ be the abelian von Neumann subalgebra of $M$
generated by the element $z$ and the unit element. Then there
exist skew-hermitian elements $a_z, b_z\in M$, depending on $z$,
such that
\begin{equation*}\label{spat}
\Delta(x)=[a_z,x] + b_z\circ x = a_z x-x a_z + \frac12 (b_z x+x
b_z)
\end{equation*}
for all $x\in \mathcal{W}^\ast(z)\subseteq H(M,\beta).$  In
particular, $\Delta$ is linear and continuous on
$\mathcal{W}^\ast(z).$
\end{enumerate}$\hfill\Box$
\end{lemma}

The results in Lemma \ref{l 2.2, 2.3 and 2.6 in one} will be now
applied to obtain a Jordan version of \cite[Proposition
2.7]{KOPR2014}. Given a
JBW$^*$-algebra $J$ whose lattice of projections is denoted by
$\mathcal{P} (J),$ and a Banach space $X$, a \emph{finitely
additive $X$-valued measure} on $\mathcal{P} (J)$ is defined in the same way as in the case of a von Neumann algebra, namely,  a mapping
$\mu: \mathcal{P} (J)\to X$ satisfying
$$\mu \left(\sum\limits_{i=1}^n p_i\right) = \sum\limits_{i=1}^{n} \mu (p_i),$$
for every family $p_1,\ldots, p_n$ of mutually orthogonal
projections in $J.$ 
\smallskip

Let $(p_i)_{i \in I}$ be a family of mutually orthogonal
projections in a JBW$^*$-algebra $J$. The series $\sum_{i\in I}
p_i$ is summable with respect to the strong$^*$ topology of $J$,
and we further know that the limit  $\displaystyle p =
\hbox{strong}^*\hbox{-}\sum_{i\in I} p_i$ is another projection in
$J$ (cf. \cite[remark 4.2.9]{HancheStor}). In particular,
$\sum_{i\in I} p_i$ is summable with respect to the weak$^*$
topology of $J$ and
$\displaystyle\hbox{strong}^*\hbox{-}\sum_{i\in I} p_i =
\hbox{weak}^*\hbox{-}\sum_{i\in I} p_i.$\smallskip

Let $J_1$ and $J_2$ be JBW$^*$-algebras, and let $\tau$ denote the
norm, the weak$^*$  or the strong$^*$ topology of $J_1.$ As in the
case of von Neumann algebras, a mapping $\mu : J_1 \to J_2$ is
said to be \emph{$\tau$-completely additive} (respectively,
\emph{countably or sequentially $\tau$-additive}) when
\begin{equation}\label{eq completely additive}
\mu\left(\sum\limits_{i\in I} p_i\right) =\tau\hbox{-}
\sum\limits_{i\in I}\mu(p_i)
\end{equation} for every family (respectively, sequence)  $\{p_i\}_{i\in I}$ of mutually orthogonal
projections in $J_1.$\smallskip

We can easily obtain now a Jordan version of \cite[Proposition
2.7]{KOPR2014}.

\begin{proposition}\label{p AyupovKudaybergenov sigma complete additivity ternary} Let $M$ be a continuous von Neumann algebra and let $\beta: M\to M$ be a $\mathbb{C}$-linear $^*$-involution.
Let $\Delta: H(M,\beta)\to H(M,\beta)$ be a {\rm(}not necessarily
linear nor continuous{\rm)} 2-local triple derivation. Then the
following statements hold:\begin{enumerate}[$(a)$]
\item The restriction $\Delta|_{\mathcal{P}(J)}$ is sequentially strong$^*$-additive, and consequently sequentially weak$^*$-additive;
\item $\Delta|_{\mathcal{P}(J)}$ is weak$^*$-completely additive, i.e.,
\begin{equation}\label{ca}
\Delta\left(\hbox{weak$^*$-}\sum\limits_{i\in I} p_i\right)
=\hbox{weak$^*$-} \sum\limits_{i\in I}\Delta(p_i)
\end{equation} for every family $(p_i)_{i\in I}$ of mutually orthogonal
projections in $J.$
\end{enumerate}
\end{proposition}

\begin{proof} $(a)$ Let $(p_n)_{n\in \mathbb{N}}$ be a sequence of mutually
orthogonal projections in $H(M,\beta).$ Let us consider the
element $z=\sum\limits_{n\in \mathbb{N}}\frac{\textstyle
1}{\textstyle n} p_n.$ Let $\mathcal{W}^\ast(z)$ be the
JBW$^*$-subalgebra of $H(M,\beta)$ generated by $z$. By Lemma
\ref{l 2.2, 2.3 and 2.6 in one}$(c)$, there exist skew-hermitian
elements $a_{z},b_{z}\in M$ with $\beta (a_z) = -a_z$ and
$\beta(b_z) = b_z$, satisfying $$T(x)=[a_{z},x] + b_{z}\circ x,$$
for all $x\in \mathcal{W}^\ast(z).$\smallskip

The elements $\sum\limits_{n=1}^{\infty} p_n$, and $ p_m$ belong
to $\mathcal{W}^\ast(z),$ for all $m\in \mathbb{N}.$ The reader
should be warned that $a_z$ might not belong to $H(M,\beta).$ In any case, the
product of $M$ is jointly strong$^*$ continuous on bounded sets,
and by \cite[Corollary]{Bunce01} $S^*(M,M_*)|_{H(M,\beta)} \equiv
S^*(H(M,\beta),H(M,\beta)_*)$. Therefore,
$$\Delta \left(\hbox{$S^*(M,M_*)$-}\sum\limits_{n=1}^{\infty} p_n\right)=\left[a_{z}, \hbox{$S^*(M,M_*)$-}\sum\limits_{n=1}^{\infty} p_n\right] +
b_{z}\circ \hbox{$S^*(M,M_*)$-}\left(\sum\limits_{n=1}^{\infty}
p_n \right)$$
$$=\hbox{$S^*(M,M_*)$-}\sum\limits_{n=1}^{\infty}[a_{z}, p_n] +
\hbox{$S^*(M,M_*)$-}\sum\limits_{n=1}^{\infty} b_{z}\circ
p_n=\hbox{$S^*(M,M_*)$-}\sum\limits_{n=1}^{\infty}  \Delta(p_n),
$$ i.e. $\Delta|_{\mathcal{P}(M)}$ is a countably or sequentially strong$^*$ additive mapping.\medskip

$(b)$ As we have commented above, the strong$^*$-topology of the
JBW$^*$-algebra $H(M,\beta)$ coincides with the restriction to
$H(M,\beta)$ of the strong$^*$-topology of $M$. When in the proof
of \cite[Proposition 2.7]{KOPR2014}$(b)$, we replace \cite[Lemmas
2.2 and 2.3]{KOPR2014} with Lemma \ref{l 2.2, 2.3 and 2.6 in one}
(and having in mind the conclusion of Proposition \ref{structure
of triple derivations on continuous type 2}), the arguments
remaind valid to obtain the desired statement here.
\end{proof}


Let $\Delta: H(M,\beta)\to H(M,\beta)$ be a {\rm(}not necessarily
linear nor continuous{\rm)} 2-local triple derivation, where $M$
is a continuous von Neumann algebra and $\beta: M\to M$ is a
$\mathbb{C}$-linear $^*$-involution. For each normal state
$\phi\in H(M,\beta)_*$ (or $\phi\in M_*$), Proposition \ref{p
AyupovKudaybergenov sigma complete additivity ternary} implies  that the
mapping  $\phi \circ \Delta|_{_{\mathcal{P}( H(M,\beta))}} :
\mathcal{P}( H(M,\beta))\to \mathbb{C}$ is a completely additive
measure. We conclude from Theorem \ref{t Jordan
Dorofeev-Sherstnev}, and from the arbitrariness of $\phi$ together
with the  uniform boundedness principle, that $
\Delta|_{_{\mathcal{P}( H(M,\beta))}} : \mathcal{P}(
H(M,\beta))\to \mathbb{C}$ is a bounded weak$^*$-completely
additive measure. An appropriate Jordan version of the
Bunce-Wright-Mackey-Gleason theorem (see Theorem \ref{t Jordan Mackey-Gleason}) implies the existence of a bounded linear operator
$G: H(M,\beta) \to H(M,\beta)$ satisfying that $G(p) = \Delta (p)$
for every $p\in  \mathcal{P}( H(M,\beta))$.\smallskip

Let us pick a self-adjoint element $z$ in $H(M,\beta)$. By Lemma
\ref{l 2.2, 2.3 and 2.6 in one}$(c)$, there exist skew-hermitian
elements $a_{z},b_{z}\in M$, with $\beta(a_z) = -a_z$ and
$\beta(b_z) = b_z$, such that $\Delta(x) = [a_{z},x]+b_{z}\circ
x,$ for every $x\in \mathcal{W}^\ast(z),$ the JBW$^*$-subalgebra
of $H(M,\beta)$ generated by $z$. Since $G|_{\mathcal{W}^\ast(z)}$
and $\Delta|_{\mathcal{W}^\ast(z)} $ are bounded linear operators
from $ \mathcal{W}^\ast(z)$ to $M$, which coincide on the set of
projections of $\mathcal{W}^\ast(z)$, and every self-adjoint
element in $\mathcal{W}^\ast(z)$ can be approximated in norm by
finite linear combinations of mutually orthogonal projections in
$\mathcal{W}^\ast(z)$, we conclude that $\Delta(x) = G(x)$ for
every $x\in \mathcal{W}^\ast(z),$ and hence
$$ \Delta(z) = G(z), \hbox{ for every } z\in H(M,\beta)_{sa}, $$
in particular, $\Delta$ is additive on
$H(M,\beta)_{sa}.$  This proves the following Proposition. \smallskip

\begin{proposition}\label{p addit no type I_n Jordan} Let $\Delta: H(M,\beta)\to H(M,\beta)$ be a {\rm(}not necessarily linear nor continuous{\rm)} 2-local triple derivation, where $M$ is a continuous von Neumann algebra and $\beta: M\to M$ is a $\mathbb{C}$-linear $^*$-involution. Then the restriction
$\Delta|_{H(M,\beta)_{sa}}$ is additive.$\hfill\Box$
\end{proposition}

\begin{lemma}\label{l 2.15 Jordan} Let $\Delta: H(M,\beta)\to H(M,\beta)$ be a {\rm(}not necessarily linear nor continuous{\rm)} 2-local triple derivation, where $M$ is a continuous von Neumann algebra and $\beta: M\to M$ is a $\mathbb{C}$-linear $^*$-involution. Suppose $\Delta(\mathbf{1})=0.$ Then there exists a skew-hermitian element $a\in M$ such that $\beta (a) = -a$, and $\Delta(x)=[a, x],$ for all $x\in H(M,\beta)_{sa}.$
\end{lemma}

\begin{proof}
Let $x\in M_{sa}.$ By Lemma~\ref{l 2.2, 2.3 and 2.6 in one}$(c)$
there exist a skew-hermitian element $a_{x,x^2}\in M$ such that
$\beta(a_{x,x^2}) =-a_{x,x^2}$, and $\Delta(x)=[a_{x,x^2},x],\,
\Delta(x^2)=[a_{x,x^2}, x^2].$

Thus, \begin{equation}\label{jord}\Delta
(x^2)=[a_{x,x^2},x^2]=[a_{x,x^2},x]x+x[a_{x,x^2},x]=2
\Delta(x)\circ x.
\end{equation}

By Proposition~\ref{p addit no type I_n Jordan} and Lemma~\ref{l
2.2, 2.3 and 2.6 in one}$(a)$, $\Delta|_{H(M,\beta)_{sa}} :
H(M,\beta)_{sa} \to H(M,\beta)_{sa}$ is a real linear mapping.
Now, we consider the linear extension $\hat{\Delta}$ of
$\Delta|_{H(M,\beta)_{sa}}$ to $H(M,\beta)$ defined by
$$ \hat{\Delta}(x_1+ix_2)=T(x_1)+i T(x_2),\, x_1, x_2\in H(M,\beta)_{sa}.$$

Taking into account the homogeneity of $\Delta,$
Proposition~\ref{p addit no type I_n Jordan} and the
identity~\eqref{jord}, we deduce that $\hat{\Delta}$ is a Jordan
$^*$-derivation (and hence, a triple derivation) on $H(M,\beta).$
Proposition \ref{structure of triple derivations on continuous
type 2} implies the existence of a skew-symmetric element $a\in M$
such that $\beta(a)=-a$ and $\hat{\Delta}(x)=[a,x]$ for all $x\in
H(M,\beta).$ In particular, $\Delta(x)=[a, x]$ for all $x\in
H(M,\beta)_{sa}$, which completes the proof. \end{proof}

We now prove the main result of this section.
\begin{theorem}\label{t continuous type 2} Let $\Delta: H(M,\beta)\to H(M,\beta)$ be a {\rm(}not necessarily linear nor continuous{\rm)} 2-local triple derivation, where $M$ is a continuous von Neumann algebra and $\beta: M\to M$ is a $\mathbb{C}$-linear $^*$-involution.  Then $\Delta$ is a linear and continuous triple derivation.
\end{theorem}

\begin{proof} From \eqref{equ delta 1 is skew} we know that $\Delta(\mathbf{1})^* = -\Delta(\mathbf{1}),$ and $M_{\Delta(\mathbf{1})} = \delta\left(\frac12 \Delta(\mathbf{1}),\mathbf{1}\right)$ is a triple derivation. Replacing $\Delta$ with $\Delta-\delta\left(\frac12 \Delta(\mathbf{1}),\mathbf{1}\right)$ we can assume that $\Delta(\mathbf{1})=0$. By Lemma \ref{l 2.15 Jordan} there exists a skew-hermitian element $a\in M$ such that $\beta (a) = -a$, and $\Delta(x)=[a, x],$ for all $x\in H(M,\beta)_{sa}.$ Observe that the mapping $\widehat{\Delta} = \Delta-[a, .]$ is a 2-local triple derivation on $H(M,\beta)_{sa}$ satisfying $\widehat{\Delta}|_{H(M,\beta)_{sa}} \equiv 0.$\smallskip

We shall finally prove that $\widehat{\Delta} =0.$ This result
follows from a direct adaptation of the arguments in \cite[Lemma
2.16]{KOPR2014}, we include here a sketch of the proof for
completeness reasons.

Let $x\in H(M,\beta)$ be an arbitrary element and let
$x=x_1+ix_2,$ where $x_1, x_2\in H(M,\beta)_{sa}.$ Since
$\widehat{\Delta}$ is homogeneous, by passing to the element
$(\mathbf{1}+\|x_2\|)^{-1} x$ if necessary, we can suppose that
$\|x_2\|<1.$ In this case the element $y=\mathbf{1}+x_2$ is
positive and invertible. Take skew-hermitian elements $a_{x,y},
b_{x,y}\in M$ such that $\beta (a_{x,y}) = -a_{x,y}$, $\beta
(b_{x,y}) = b_{x,y},$ and
$$\widehat{\Delta}(x)=[a_{x,y},x]+b_{x,y}\circ x, \hbox{ and } \widehat{\Delta}(y)=[a_{x,y},y]+b_{x,y}\circ y.$$
Since $\widehat{\Delta}(y)=0,$ we get $[a_{x,y},y]+b_{x,y}\circ
y=0.$ Lemma 2.4 in \cite{KOPR2014} implies that $[a_{x,y}, y]=0$
and $ib_{x,y} \circ y=0.$ Having in mind that $y$ is positive and
invertible, and that $ib_{x,y}$ is hermitian, \cite[Lemma
2.5]{KOPR2014} proves that $b_{x,y}=0.$\smallskip

The condition $0=[a_{x,y},y]=[a_{x,y}, \mathbf{1}+x_2]=[a_{x,y},
x_2],$ implies $$\widehat{\Delta} (x)=[a_{x,y},x]+b_{x,y}\circ
x=[a_{x,y}, x_1+ix_2]=[a_{x,y}, x_1],$$ which shows that
$$
\widehat{\Delta}(x)^\ast=[a_{x,y}, x_1]^\ast=[x_1,
a_{x,y}^\ast]=[x_1, -a_{x,y}]=[a_{x,y},
x_1]=\widehat{\Delta}(x).$$ The arbitrariness of $x\in H(M,\beta)$
implies that $\widehat{\Delta}(x)=0$, as desired.
\end{proof}

Since every element in a closed ideal of a JB$^*$-triple can be written as a cube of an element in that ideal, it is clear that a triple derivation leaves closed ideals invariant.  Hence the same is true for 2-local triple derivations.  Thus, by invoking the structure theorem of continuous JBW$^*$-triples stated at the beginning of this section, and combining Theorems~\ref{triple-rectangular} and \ref{t continuous type 2}, we obtain the second main result of this paper.
\begin{theorem}\label{continuous} Let $\Delta: A\to A$ be a {\rm(}not necessarily linear nor continuous{\rm)} 2-local triple derivation, where $A$ is a continuous JBW$^*$-triple.  Then $\Delta$ is a linear and continuous triple derivation.
\end{theorem}

\begin{problem}\label{5.10}
Does Theorem~\ref{t Jordan Dorofeev-Sherstnev} remain valid when $H(M,\beta)$ is replaced by an arbitrary JBW$^*$-algebra without summands of type $I_n$?
\end{problem}

\begin{problem}\label{5.9}
Is Theorem~\ref{continuous} valid for
\begin{description}
\item[(a)] JBW$^*$-triples of type I? {\rm (See Corollary~\ref{Cartan})}
\item[(b)]  reversible JBW$^*$-algebras?
\item[(c)] 2-local triple derivations with values in a Jordan triple module?
\item[(d)] 2-local triple derivations on various algebras of measurable operators?
\item[(e)] real JBW$^*$-triples?
\item[(f)] complex and real JB$^*$-triples?
\end{description}
\end{problem}


\begin{thebibliography}{22}

\bibitem{Ayu1989} S.\  Ayupov, A new proof of the existence of traces on Jordan operator algebras and real von Neumann algebras, \emph{J.\  Funct.\  Anal.} \textbf{84}, no. 2, 312-321 (1989).

 \bibitem{AyuKud} S.\ Ayupov,  K.\  Kudaybergenov,
Derivations, local, and 2-local derivations on algebras of measurable operators, In: Proceedings of the USA-Uzbekistan conference on Analysis and Mathematical Physics, May 20-23, 2014, Fullerton, California, \emph{Contemporary Mathematics}, to appear.

\bibitem{AyuKudPos} S.\ Ayupov, K.\  Kudaybergenov, 2-local derivations on von Neumann algebras. \emph{Positivity} \textbf{19}, no. 3, 445--455 (2015).

\bibitem{AyuKudPer} S.\ Ayupov, K.\ Kudaybergenov,  A.\ M.\ Peralta,
 A survey on local and 2-local derivations on C*-algebras and von Neumann algebras, In: Proceedings of the USA-Uzbekistan conference on Analysis and Mathematical Physics, May 20-23, 2014, Fullerton, California, \emph{Contemporary Mathematics}, to appear.


\bibitem{AyuRakUsma97} S.\  Ayupov, A.\ Rakhimov, S.\ Usmanov: Jordan, \emph{Real and Lie Structures in Operator Algebras}. Kluwer, 1997.

\bibitem{Black2006} B. Blackadar, \emph{Operator algebras. Theory of C*-algebras and von Neumann algebras.} Encyclopaedia of Mathematical Sciences, 122. Operator Algebras and Non-commutative Geometry, III. Springer-Verlag, Berlin, 2006.

\bibitem{BHKArXiv} M.\ Bohata, J,\ Hamhalter, O.\ F.\ K.\  Kalenda,
Decompositions of preduals of JBW and JBW$^*$-algebras, arXiv:1511.01086 (2015)

\bibitem{Bunce01} L.\ J.\ Bunce, Norm preserving extensions in JBW*-triple preduals, \emph{Quart.\  J.\  Math.} \textbf{52} (2), 133-136 (2001).

\bibitem{BuWri89} L.\ J.\  Bunce, J.\ D.\ M.\  Wright, Continuity and linear extensions of quantum measures on Jordan
operator algebras, \emph{Math.\  Scand.} \textbf{64}, 300-306 (1989).

\bibitem{BuWri92} L.\ J.\  Bunce, J.\ D.\ M.\  Wright,
The Mackey-Gleason problem, \emph{Bull.\  Amer.\  Math.\  Soc.} \textbf{26}, 288-293 (1992).

\bibitem{BuWri94} L.\ J.\  Bunce, J.\ D.\ M.\  Wright,
The Mackey-Gleason problem for vector measures on projections in von Neumann algebras,
\emph{J.\  London Math.\  Soc.} \textbf{49}, 133-149 (1994).

\bibitem{BurPolPerBLMS14}
M.\ Burgos, F.\  J.\ Fern\'andez-Polo, A.\ M.\ Peralta,  Local triple derivations on C*-algebras and JB*-triples. \emph{Bull.\  Lond.\  Math.\  Soc.} \textbf{46}, no. 4, 709--724  (2014).

\bibitem{BuFerGarPe2015RACSAM} M. Burgos, F.J.  Fern{\'a}ndez-Polo, J.J.  Garc{\'e}s and A.M. Peralta, A Kowalski-Slodkowski theorem for 2-local $*$-homomorphisms on von Neumann algebras, \emph{Rev. R. Acad. Cienc. Exactas Fís. Nat. Ser. A Math. RACSAM} \textbf{109}, no. 2, 551-568 (2015).

\bibitem{BuFerGarPe2015JMAA} M. Burgos, F.J. Fern{\' a}ndez-Polo, J.J.
Garc{\'e}s, A.M. Peralta, 2-local triple homomorphisms on von Neumann algebras and JBW$^*$-triples, \emph{J. Math. Anal. Appl.} \textbf{426}, 43-63 (2015).


\bibitem{chu} C.-H.\ Chu,  \emph{Jordan structures in geometry and analysis}. Cambridge Tracts in Mathematics, 190. Cambridge University Press, Cambridge, 2012. x+261 pp

\bibitem{CooJunNavPerVill2013} T. Cooney, M. Junge, M. Navascués, D. Pérez-García, I. Villanueva, Joint system quantum descriptions arising from local quantumness, \emph{Comm. Math. Phys.} \textbf{322}, no. 2, 501-513 (2013).

\bibitem{Doro} S.\  Dorofeev, On the problem of boundedness of a signed measure on projections of a von Neumann algebra, \emph{J.\  Funct.\ Anal.} \textbf{103}, 209-216 (1992).
    
\bibitem{Dvu1993} A. Dvurecenskij, \emph{Gleason's theorem and its applications}. Mathematics and its Applications (East European Series), 60. Kluwer Academic Publishers Group, Dordrecht; Ister Science Press, Bratislava, 1993. ISBN: 0-7923-1990-7

\bibitem{EdRu98} C.\ M.\  Edwards, G.\ T.\  R\"{u}ttimann, Exposed faces of the unit ball in a JBW$^*$-triple, \emph{Math.\  Scand.} \textbf{82}, 287-304  (1998).

\bibitem{EdRu99} C.\ M.\  Edwards, G.\ T.\  R\"{u}ttimann, Gleason's theorem for rectangular JBW$^*$-triples, \emph{Comm. Math. Phys.} \textbf{203}, no. 2, 269-295  (1999).


\bibitem{Gleason57} A.\  M.\  Gleason, Measures on the closed subspaces of a Hilbert space, \emph{J.\
Math.\  Mech.} \textbf{6}, 885-893 (1957) .


\bibitem{Ham03} J.\ Hamhalter, \emph{Quantum Measure Theory}, Kluwer, 2003.

\bibitem{HancheStor} H.\  Hanche-Olsen, E.\  St{\o }rmer, \textit{Jordan operator algebras}, Monographs and Studies in Mathematics 21, Pitman, London-Boston-Melbourne 1984.

\bibitem{HoMarPeRu} T.\  Ho, J.\  Mart\'inez-Moreno, A.\ M.\  Peralta, B.\  Russo,
 Derivations on real and complex JB$^\ast$-triples, \emph{J.\  London Math.\  Soc.} (2)
 \textbf{65}, no. 1, 85-102 (2002).

\bibitem{Horn87}
G.\ Horn, Classification of JBW$^*$-triples of type I. \emph{Math.\  Z.} 196 (1987), no. 2, 271-291.

\bibitem{Horn88} G.\  Horn, E.\  Neher, Classification of continuous JBW$^\ast$-triples, \textit{Trans.\ Amer.\
Math.\  Soc.} \textbf{306}, 553-578 (1988).

\bibitem{John01} B.\ E.\  Johnson, Local derivations on C$^*$-algebras are derivations,
\emph{Trans.\  Amer.\  Math.\  Soc.} \textbf{353}, 313-325 (2001).


\bibitem{Kad90} R.\ V.\  Kadison, Local derivations, \emph{J.\  Algebra} \textbf{130}, 494-509 (1990).


\bibitem{KR86} R.\ V.\  Kadison, J.\  Ringrose, \emph{Fundamentals of the theory
of operator  algebras}, Volume II, Birkhauser,  Boston, 1986.


\bibitem{KR92} R.\  V.\  Kadison, J.\  Ringrose, \emph{Fundamentals of the theory
of operator  algebras}, Volume IV, Birkhauser,  Boston, 1992.

\bibitem{kaup} W.\ Kaup, A Riemann mapping theorem for bounded symmetric domains in complex Banach spaces. \emph{Math.\  Z.} \textbf{183}, no. 4, 503-529 (1983).

\bibitem{KOPR2014}
K.\ Kudaybergenov, T.\ Oikhberg, A.\ M.\ Peralta, B.\ Russo,
2-Local triple derivations on von Neumann algebras,
\emph{Illinois J.\  Math.} \textbf{58}, no. 4, 1055--1069 (2014).

\bibitem{LarSou} D.\ R.\  Larson, A.\ R.\  Sourour, Local derivations and local automorphisms of $B(X)$, \emph{Proc.\ Sympos.\ Pure Math.} \textbf{51}, Part 2, Providence, Rhode Island 1990, pp. 187-194.

\bibitem{Mack} M.\ Mackey, Local derivations on Jordan triples, \emph{Bull. London Math.\ Soc.} \textbf{45}, no. 4, 811-824 (2013). doi: 10.1112/blms/bdt007

\bibitem{Matv} M.\ S.\  Matve\u{i}chuk, Linearity of a charge on a lattice of orthoprojectors. (Russian)  \emph{Izv.\  Vyssh.\  Uchebn.\  Zaved.\ Mat.} 1995, no. 9, 48-66; translation in \emph{Russian Math.\  (Iz.\  VUZ)} \textbf{39}, no. 9, 46-64  (1995).

\bibitem{Molnar2001} L. Molnár, Local automorphisms of some quantum mechanical structures, \emph{Lett. Math. Phys.} \textbf{58}, no. 2, 91-100 (2001).

\bibitem{RuWrigh00} O. Rudolph, J.D.M. Wright, On unentangled Gleason theorems for quantum information theory, \emph{Lett. Math. Phys.} \textbf{52}, no. 3, 239-245 (2000).


\bibitem{Semrl97}  P.\  \v{S}emrl, Local automorphisms and derivations on $B(H)$, \emph{Proc. Amer.\  Math.\  Soc.} \textbf{125}, 2677-2680 (1997).


\bibitem{Shers2008} A.\ N.\  Sherstnev, \emph{Methods of bilinear forms in noncommutative theory of measure and integral}, Moscow, Fizmatlit,  2008,
256 pp.

\bibitem{Tak} Takesaki, M., \emph{Theory of operator algebras I}, Springer
Verlag, New York,  1979.


\bibitem{To} D.\ M.\ Topping, \emph{Jordan algebras of self-adjoint operators}, Mem.\  Amer.\  Math.\  Soc., \textbf{53}, 1965. 48pp


\bibitem{Upmeier80} H.\  Upmeier, Derivations of Jordan C$^*$-algebras, \emph{Math.\  Scand.}
 \textbf{46}, 251-264 (1980).

\bibitem{Wright98} J.D.M. Wright, Decoherence functionals for von Neumann quantum histories: boundedness and countable additivity, \emph{Comm. Math. Phys.} \textbf{191}, no. 3, 493-500 (1998).


\end{thebibliography}
\end{document}